\documentclass[12pt]{amsart}
\usepackage{graphicx}
\usepackage{amsmath}
\usepackage{amstext}
\usepackage{amsfonts}
\usepackage{amsthm}
\usepackage{amssymb}

\DeclareMathOperator{\image}{im}

\newtheorem{theorem}{Theorem}
\newtheorem{definition}{Definition}

\newtheorem{lemma}{Lemma}
\newtheorem{corollary}{Corollary}

\newcommand{\Ar}{\mathbb{R}}

\newcommand{\ep}{\epsilon}

\newcommand{\term}[1]{\textbf{#1}}

\newcommand{\twist}{J}

\newcommand{\GF}{\mathcal{F}}
\newcommand{\C}{\mathbb{C}}
\newcommand{\Hep}{H^{+}}
\newcommand{\Hem}{H^{-}}
\newcommand{\Hei}{H^{\infty}}
\newcommand{\Cep}{C^{+}}
\newcommand{\Cem}{C^{-}}
\newcommand{\Cei}{C^{\infty}}
\newcommand{\HHep}{H_{+}}
\newcommand{\HHem}{H_{-}}
\newcommand{\HHei}{H_{\infty}}

\newcommand{\Crease}{Cre}
\newcommand{\partiale}{\partial_\twist}

\begin{document}

\title{Geometric Homology}
\author{Max Lipyanskiy}
\date{}

\address{Simons Center for Geometry and Physics, Stony Brook University, Stony Brook, NY 11794}
\email{mlipyan@math.columbia.edu}
\maketitle

\section{Introduction}
The purpose of this paper is to introduce a version of singular homology based on smooth mappings of manifolds with corners.  Although variants of such a theory  exists in the literature (see \cite{Zinger}, \cite{Kahn}, \cite{Keck}), we felt that certain points were not adequately addressed.  In particular, our goal is to construct a chain level theory based on smooth mappings of manifolds with corners.  In addition, we will avoid using the fact that smooth manifolds with can be triangulated.  As we shall see, transversality and intersections play a major role in setting up this theory.  From a pedagogical viewpoint, having intersection theory arguments available from the start facilitates simple and intuitive computations. \\\\
The motivation for our construction comes from several sources.  There has been a renewed interest in defining invariants based on geometric chains mapping to configuration space.  Such chains appear as solutions of various nonlinear partial differential equations in Gromov-Witten theory as well as gauge theory.  In addition, geometric chains also appear in the work of Sullivan and Chas on String topology \cite{Sullivan}.  The geometric construction of various $\infty$-structures is closely related to a chain level construction of intersection theory. \\\\
Ultimately, however, we are interested in the infinite dimensional variant of this theory.  In \cite{Lipyan}, we introduced a new approach to Floer theory based on a notion of a semi-infinite cycle.  This approach provides an alternative (but equivalent) approach to the Morse theory construction of Floer homology.  The fundamental objects in this theory are semi-infinite chains which are mappings of Hilbert manifolds $$\sigma: P\rightarrow B$$ meeting certain topological axioms.  The theory we discuss in this paper is a finite dimensional analogue of this geometric construction.  As it turns out, many arguments in Floer theory (for instance, the proof of the Morse Homology theorem) are nearly identical to their finite dimensional counterparts once certain topological assumptions are satisfied.  For these reasons, we thought it would be useful to expose our approach to geometric homology in the finite dimensional case to help motivate the analogous arguments in Floer theory.\\\\

\textbf{Acknowledgement.}  We wish to thank Tom Mrowka and Dennis Sullivan for useful conversations.  In addition, we would like to thank the Simons Center For Geometry and Physics for their hospitality while this work was being completed.

\section{Geometric Preliminaries}
\subsection{Manifolds With Corners}
Here we review the concept of a manifold with corners and prove various geometric lemmas that are used in the paper.  
\begin{definition}
A map $\Ar^k\times [0,1)^j\rightarrow \Ar^l\times [0,1)^k$ is \textbf{smooth} if locally it is a restriction of a smooth map on $\Ar^{k+j}\rightarrow \Ar^{l+k}$. 
\end{definition}
\begin{definition}
Let $P$ be a Hausdorff, 2nd countable topological space.  $P$ is a \textbf{smooth manifold with corners} if it is equipped with an open cover $\{ U_i \}$ together with  homeomorphisms $$\phi_i: \Ar^k\times [0,1)^{n-k}\rightarrow U_i$$
We demand that $\phi_j^{-1}\circ \phi_i$ are smooth where defined. The collection $(U_i,\phi_i)$ is an \textbf{atlas} for $P$.  As usual,  picking a maximal compatible atlas specifies a smooth structure on $P$.  
\end{definition}
\noindent $P$ has a natural structure of a stratified space. Indeed, Let $P^i$ be the set of points $p$ with a local chart of the form $\Ar^{n-i}\times [0,1)^{i}$ where $p=\{0 \}$.  We have $$P=P^0\cup P^1\cup P^2\cup P^3 \dots$$  One may check, using the inverse function theorem, that definition of strata is independent of the choice of chart. Note that $P^i$ is a manifold without corners of dimension $n-i$.  We will assume that, for a given $P$,  each connected component of the top stratum has the same dimension $n$.  Let $M$ be a smooth manifold without boundary and $f:P\rightarrow M$ be a smooth map.
\begin{definition}
 A map $f:P\rightarrow M$ is \textbf{transverse} to a submanifold $Y\subset M$ if its transverse on each open stratum $P^i$ of $P$.
\end{definition}
\begin{lemma}
Consider a  manifold $M$ with closed submanifold $Y$.  Given a smooth map $$\sigma:P \rightarrow M$$ such that $\sigma$ is transverse to $Y$, $\sigma^{-1}(Y)$ is a manifold with corners.
\end{lemma}
\begin{proof}
Locally, we may write $\sigma$ as $$\sigma:\Ar^l \times [0,1)^i\rightarrow \Ar^n$$ with $Y=\Ar^k\times 0 \subset \Ar^n$.  Let $p\in \Ar^l$.  We will produce a chart for $\sigma^{-1}(Y)$  near $p\times 0$.  Let $\pi:\Ar^n\rightarrow \Ar^{n-k}$ be the projection to the last $n-k$ factors.  The map $$f: \Ar^l \times [0,1)^i\rightarrow \Ar^{n-k}\times [0,1)^i$$ given by $f(v,t)=(\pi(\sigma(v,t)),t)$ is a surjection in view of our transversality assumption.  The implicit function theorem assures that we can pick coordinates $U\times [0,1)^i$ for around $p$ such that $\pi(\sigma(u,t))=\phi(u)$ for some smooth function $\phi$.  Therefore, another application of the implicit function theorem yields that $\phi^{-1}(0)\times [0,1)^i$ is a chart for $\sigma^{-1}(Y)$.        
\end{proof}

\subsection{Boundary Operator}
Our goal is to define a geometric boundary operator on manifolds with corners. For $[0,1)^k$ let $$\partial_i([0,1)^k)= [0,1)^{k-1}$$ with the  natural face inclusion $$s_i:\partial_i([0,1)^k)\rightarrow [0,1)^k$$ that omits the $i$th factor.  For $[0,1)^k$, let $$\partial [0,1)^k=\bigsqcup_i \partial_i [0,1)^k$$

Given a $P$, we construct $\partial P$ as follows.  We replace the chart $\Ar^{n-k} \times [0,1)^k$ by $\Ar^{n-k} \times \partial [0,1)^k$.  The transition functions for $P$ induce transition functions for $\partial P$.  Note that locally, we get one copy of $P^1$, two copies of $P^2$, etc.

\subsection{Orientations}

\begin{definition}
$P$ is said to be \term{orientable} if $P^0$ is orientable.  An \term{orientation} of $P$ is an orientation of $P^0$.
\end{definition}

The following definition is important for the construction of homology groups:
\begin{definition}
Given a manifold with corners $P$, a mapping $\sigma:P\rightarrow M$ is said to be \textbf{trivial}, if there exists an orientation reversing diffeomorphism $f:P\rightarrow P$ with $\sigma\circ f= \sigma$.
\end{definition}

\begin{lemma}
An orientation of $P$ induces an orientation of $\partial P$.
\end{lemma}
\begin{proof}
Since $(\partial P)^0=P^1$, this is a consequence of the fact that an orientation of $P^0$ induces an orientation on $P^1$ using the ``outward normal first'' convention.   
\end{proof}
Since diffeomorphism of $P$ induces one of $\partial P$, we have:
\begin{lemma}
If $\sigma: P\rightarrow M$ is trivial, so is $\partial \sigma:\partial P\rightarrow M$.
\end{lemma}

The following lemma will imply that $\partial^2=0$ in the homology theory we define:
\begin{lemma}
 Given any $\sigma$, we have $\partial^2\sigma$ trivial.
\end{lemma}
\begin{proof} Locally, $\partial^2 (V\times [0,1)^k)=V\times \partial^2[0,1)^k$.  Now, $$\partial^2[0,1)^k=\coprod_{i\neq j} \partial_i\circ \partial_j [0,1)^k$$  Thus, $\partial^2[0,1)^k$ is naturally a disjoint union of two spaces corresponding to when $i<j$ and $i>j$. This provides the orientation reversing involution of $\partial^2 \sigma$.
\end{proof}

We now recall the notion of orientable map between finite dimensional manifolds.  For the general discussion, see \cite{DK}.  Given a smooth map $$\sigma: P \rightarrow M$$ of finite dimensional manifolds, recall the construction of determinant bundle $$det(\sigma)\rightarrow P$$ Intuitively, the line over a point $p\in P$ corresponds to the space $$\Lambda^{max}Ker D\sigma_p \otimes \Lambda^{max} (Coker D\sigma_p)^*$$  Since in general the dimension of $kerD\sigma $ may jump, to construct a locally trivial bundle one proceeds as follows. Locally, one may pick a trivialized bundle $\underline{\Ar^n}  \rightarrow P$ and a map of bundles $$L: \underline{\Ar^n}  \rightarrow \sigma^*(TM)$$  such that the map $D\sigma \oplus L: TP \oplus \underline{\Ar^n}  \rightarrow \sigma^*(TM)$ is surjective.  One then defines $det(\sigma)$ as $$\Lambda^{max} (Ker(D\sigma \oplus L))\otimes (\Lambda ^{n}\underline{\Ar^n})^*$$  One can check that does not depend on the local choices and defines a line bundle over $P$.  A useful fact in dealing with orientations is that any exact sequence:
$$0\rightarrow V_0 \rightarrow V_1 \rightarrow \cdots  V_n \rightarrow 0$$
gives rise to a canonical isomorphism $$\otimes_{even} \Lambda^{max} V_i \cong \otimes_{odd} \Lambda^{max}V_i$$

\begin{definition}
An orientation of a map $\sigma$ is an orientation of $\det(\sigma)$.  
\end{definition}
We will need several basic results about oriented maps.  Given transverse maps $\sigma:P\rightarrow N$ and $f:M\rightarrow N$, we may form the pullback $$f^*(\sigma):f^{-1}(P)\rightarrow M$$
This is defined as the subset of $M\times P$ such that $(m,p) \in f^{-1}(P)$ if $f(m)=\sigma(p)$.

\begin{lemma}
We have $f^*(det(\sigma))\cong det(f^*(\sigma))$
\end{lemma}
\begin{proof}
The tangent space to $f^{-1}(P)$ is given by $(v,w)\in TM\times TP$ with $Df(v)=D\sigma(w)$.  The kernel of the map $f^{*}\sigma:f^{-1}(P)\rightarrow M$ is given by points $(v,w)$ with $v=0$ and $D\sigma(w)=0$.  Thus, $ker(f^{-1}\sigma)$ is naturally identified with $ker(\sigma)$.  In view of the isomorphism $$TM/im(f^{-1}\sigma)\rightarrow TN/im(\sigma)$$ the cokernels are also naturally identified.  
\end{proof}
This implies that an orientation of the map $\sigma$ induces one on $f^*(\sigma)$.
\begin{lemma}
\label{Lem8}
 If $f$ is oriented and $P$ is oriented, so is $f^{-1}(P)$.
\end{lemma}
\begin{proof}
The tangent space to $f^{-1}(P)$ is given by $(v,w)\in TM\oplus TP$ with $Df(v)=D\sigma(w)$.  The exact sequence $$0\rightarrow Ker Df\rightarrow TM \rightarrow TN \rightarrow Coker Df \rightarrow 0$$ gives a trivialization of $\Lambda^{max} TM \otimes \Lambda^{max} (TN)^*$ while the sequence $$0\rightarrow Tf^{-1}(P) \rightarrow TM\oplus  TP \rightarrow TN \rightarrow 0$$ implies than a trivialization of $\Lambda^{max} Tf^{-1}(P)$ is equivalent to that of $\Lambda^{max} TM \otimes \Lambda^{max} (TN)^* \otimes \Lambda^{max} TP$. 
\end{proof}
Note that given an oriented map $\sigma:P\rightarrow M$, we get an induced oriented map $\partial \sigma: \partial P\rightarrow M$.  The following is straighforward to verify using a local chart:
\begin{lemma}
Given an oriented map $\sigma:P\rightarrow M$, there exists a diffeomorphism $$\phi: \partial^2 P \rightarrow \partial^2 P$$ commuting with $\sigma$ such that $$\phi^*(det(\sigma))\cong -det(\sigma)$$ as oriented line bundles. 
\end{lemma}

\subsection{Cutting by a Hypersurface}
We will need to cut a manifold by a hypersurface obtaining a decomposition of the manifold into two parts.  Let $$\sigma : P\rightarrow M$$ be as above. Consider a smooth map $$f:M\rightarrow \Ar$$ with $p$ a regular value of $f\circ \sigma$.  Let $$\sigma^+=\sigma_{|(f\circ \sigma)^{-1}[p,\infty)}$$ $$\sigma^-=\sigma_{|(f\circ \sigma)^{-1}(-\infty,p]}$$ and $$\sigma^0=\sigma_{|(f\circ \sigma)^{-1}(p)}$$  Let $$P'=(f\circ \sigma)^{-1}(p)$$
\begin{theorem}
    $\sigma^{\pm}$ are manifolds with corners.  Furthermore, $$\partial(\sigma^{\pm})=(\partial \sigma)^{\pm}\sqcup \pm \sigma^{0}$$
\end{theorem}
\begin{proof}
Locally we have $$\sigma:V\times [0,1)^k \rightarrow M$$ where $V$ is a manifold without boundary.  Applying the inverse function theorem, we can represent $\sigma^{\pm}$ as $(\sigma_{|V})^{\pm}\times [0,1)^k$.  The formula for the boundary follows since $$\partial(\sigma_{|V}^{\pm}\times [0,1)^k)=\partial \sigma_{|V}^{\pm}\times [0,1)^k\sqcup \sigma_{|V}^{\pm}\times \partial[0,1)^k$$ and $$\partial \sigma_{|V}^{\pm}=\pm \sigma_{|V}^{0}$$
\end{proof}

We also need to prove a result that will allows us to cut cycles into smaller pieces.  We  introduce a manifold $\Crease(\sigma)$ that interpolates between $\sigma$ and "creasing" $\sigma$ along $\sigma^{0}$. Let $D$ be subset of $\Ar^2$ given by $$\{(x,y)\in \Ar^2 | 0\leq y \leq 2-|x|,|x|<1\}$$  $D$ inherits the structure of a manifold with corners as a subset of the plane.  We define a homeomorphism  $\phi:D \rightarrow (-1,1)\times [0,1]$.  Let $$\phi(x,y)=(x,\frac{y}{2-|x|})$$  Note that $\phi$  is a diffeomorphism outside $x=0$. 

Let $\Crease(\sigma)$ be obtained as follows. As a topological space, $$\Crease(\sigma)=P\times [0,1]$$  Now, we specify the stratification. Outside, $P'\times [0,1]$, let the manifold structure be given by the product structure $(P-P')\times [0,1]$. Near a point on $P'$, we can locally write $P$ as $P=V\times (-\ep,\ep)$ with $f(\sigma(v,t))=t$ and $V$ is a manifold with corners. We take the manifold structure to be $V\times D_\ep$ where $$D_\ep=\{(x,y)\in D|=-\ep<x<\ep\}$$  Here we identify  $P=V\times (-\ep,\ep)\subset V\times D_\ep$ as points with $y=0$.  The overlap chart map $$V\times D_\ep \rightarrow V\times (-\ep,\ep)\times [0,1]$$ is induced by restricting $\phi$ to $$\phi_\ep:D_\ep \rightarrow (-\ep,\ep)\times [0,1]$$  
Since $\sigma$ induces a map $\Crease(\sigma)$ by composing $\sigma$ with the smooth projection $$V\times D_\ep\rightarrow V\times (-\ep,\ep)$$ 
we have
\begin{lemma}
$\sigma$ induces smooth map on $\Crease(\sigma)$.  Furthermore, if $\sigma$ is trivial, so is $\Crease(\sigma)$.
\end{lemma}
 
\begin{lemma}
$\partial(\Crease(\sigma))=-\sigma \sqcup \sigma^{+}\sqcup\sigma^{-}\sqcup \Crease(\partial \sigma)$
\end{lemma}
\begin{proof}
We check the statement in a neighborhood of $P'$.  We have the chart $V\times D_\ep$.   We have $$\partial(V\times D_\ep)=\partial V\times D_\ep \sqcup V\times \partial D_\ep$$ while $$\partial D_\ep=(-\ep,\ep)\sqcup (-\ep,0]\sqcup [0,\ep)$$
$\partial V\times D_\ep$  corresponds to $\Crease(\partial P)$\\\\ $V\times (-\ep,\ep)$ corresponds to $P$\\\\
$V\times [0,\ep) $ corresponds to $\sigma^{+}$ and $V\times (-\ep,0] $ to $\sigma^{-}$.
\end{proof}

\section{Definition of the Homology Groups} 
\label{homdef}
Fix some countable infinite dimensional Hilbert space $\mathbb{H}$ once and for all.  To avoid set-theoretic complications, all our chains will be subsets of $\mathbb{H}$. Let $M$ be a smooth paracompact manifold without boundary.  We wish to construct a homology theory for $M$ based on mappings of manifolds with corners into $M$.  Although we focus on the case $M$ is finite dimensional, most of the methods of this paper carry over rather directly to the case when $M$ is an infinite dimensional paracompact Banach  manifold. 
\begin{definition}
A \term{chain} is a smooth map $$\sigma:P\rightarrow M$$ where $P$ is a compact oriented manifold with corners embedded in $\mathbb{H}$.  Two chains $\sigma:P\rightarrow M$ and $\tau:Q\rightarrow M$ are said to be \term{isomorphic} if there exists an orientation preserving diffeomorphism $f:P\rightarrow Q$ such that $\tau\circ f=\sigma$.
\end{definition}
We will assume that all the components of $P$ have the same dimension.
\begin{definition}
A chain $\sigma: P\rightarrow M$ is said to be \term{trivial} if there exists an orientation reversing diffeomorphism $f:P\rightarrow P$ with $\sigma\circ f=\sigma$.
\end{definition}
\begin{definition}
A chain $\sigma:P\rightarrow M$ is said to have \term{small image} if $\sigma(P)\subset g(N)$ where $g:N\rightarrow M$ is a smooth map from a manifold with corners (not necessarily compact) with $dim(N)<dim(P)$.
\end{definition}
\begin{definition}
A chain $\sigma$ is said to be \term{degenerate} if $\sigma$ has small image and $\partial \sigma$ is isomorphic to a disjoint union of a trivial chain and a chain with small image.
\end{definition}
As the simplest example, note that the constant map $[0,1]\rightarrow pt$ is degenerate but the boundary does not have small image.  
\begin{lemma}
Let $\tau$ be a trivial chain.  If $\sigma \sqcup \tau$ is trivial, then $\sigma$ is trivial. \label{lem32}
\end{lemma}
\begin{proof}
We decompose $\sigma$ into mutually isomorphic components as $$\sigma=\sigma_1\sqcup \sigma_2 \dots $$ Here, each $\sigma_i$ consists of the disjoint union of isomorphic connected components of (up to orientation) of $\sigma$ and no $\sigma_i$, $\sigma_j$ share isomorphic components when $i\neq j$.  An automorphism of $\sigma$ preserved these components.  Therefore, $\sigma$ is trivial exactly when for each i, either the number of components of $\sigma_i$ is zero when counted with orientation or each such component admits an orientation reversing isomorphism.  Since $\tau$ is trivial, $\sigma \sqcup \tau$ either adds components which are trivial or adds zero components when counted with sign.  Therefore, the only way $\sigma \sqcup \tau$ is trivial is if $\sigma$ was already trivial.
\end{proof}
\begin{definition}
Let $Q(M)$ be the set of chains isomorphic to $\alpha \sqcup \beta$ where $\alpha$ is trivial and $\beta$ is degenerate.  We allow $\alpha$ and $\beta$ to be empty.
\end{definition}

\begin{lemma}
If $\sigma$ is in $Q(M)$, so is $\partial \sigma$.  
\end{lemma}
\begin{proof}
Since the  boundary of a trivial chain is trivial, we can focus on when $\sigma$ is degenerate.  We have $\partial \sigma=\alpha \sqcup \beta$ with $\alpha$ trivial and $\beta$ with small image.  We need to show that $\beta$ is degenerate. Since $\partial^2 \sigma $ is trivial, we have $\partial \alpha \sqcup \partial \beta$ trivial.  By the previous lemma $\partial \beta$ is trivial, hence $\beta$ is degenerate as desired.   
\end{proof}
The following is the key step to defining an equivalence relation based on chains in $Q(M)$:
\begin{lemma}
If $\sigma\sqcup \tau \in Q(M)$ for some $\tau \in Q(M)$, we have $\sigma \in Q(M)$.  
\end{lemma}
\begin{proof}
We decompose $\sigma$ as $\sigma_1\sqcup \sigma_2 \dots $ as in the  lemma above.  The hypothesis implies that each $\sigma_i$ is either trivial or has small image.  Thus, we can write $\sigma=\alpha \sqcup \beta$ with $\alpha$ trivial and $\beta$ having small image.  We show that $\beta$ is degenerate.  We have $\partial \sigma \sqcup \partial \tau \in Q(M)$.  By repeating the argument, this implies that $\partial \beta$ is a union of a small chain and a trivial chain.  Therefore, $\beta$ is degenerate.   
\end{proof}

We are ready to define our chain complex:
\begin{definition}
Define an equivalence on chains as follows:  $\sigma \sim \tau$ if $\sigma \sqcup -\tau$ is in $Q(M)$.  We denote the resulting set by $C_*(M)$.
\end{definition}
\begin{lemma}
$\sim$ is an equivalence relation.  The geometric boundary operator $\partial$ induces the structure of a chain complex on $C_*(M)$ with addition given by disjoint union.
\end{lemma}
\begin{proof}
Reflexivity and symmetry are clear.  To check transitivity, note that $\sigma \sqcup -\tau \in Q(M)$ and  $\tau \sqcup -\rho \in Q(M)$ imply $\sigma \sqcup -\tau \sqcup \tau \sqcup -\rho \in Q(M)$.  By the previous lemma, $\sigma \sqcup -\rho$ is in $Q(M)$ as desired.    The additive structure is induced by disjoint union: Given $\sigma:P\rightarrow M$ and $\tau:Q\rightarrow M$ we must choose an embedding of $P\sqcup Q$ in our Hilbert space $\mathbb{H}$.  However, any two such choices will be equivalent under our relation since the difference will be isomorphic to a trivial chain. $C_*(M)$ has inverses since $\sigma \sqcup -\sigma$ is trivial for any $\sigma$. Finally, $C_*(M)$ forms a complex since $\partial Q(M)\subset Q(M)$ and $\partial^2 \sigma$ is always trivial.  
\end{proof}
Note that $C_*(M)$ has a natural grading given by the dimension of the chains. 
\begin{definition}
Let $H_*(M)$ denote the homology groups associated to the complex $C_*(M)$.
\end{definition}
Let us observe that cycles with small image are $0$ in $C_*(M)$:
\begin{lemma}
If $\sigma$ has small image and $\partial \sigma \sim 0$, then $\sigma \sim 0$. 
\end{lemma}
\begin{proof}
We have that $\partial\sigma \in Q(M)$. Therefore, $\sigma$ is degenerate and hence also in $Q(M)$. 
\end{proof}

As will be proved in subsequent sections, the homology groups introduced above are isomorphic to the singular homology of $M$. 

\section{Cohomology}
There is a variant of the groups that is based on proper rather than compact chains.  These will give a geometric version of cohomology.  
\begin{definition}
A \term{cochain} is a smooth oriented map $\sigma:P\rightarrow M$ where $P$ is a smooth manifold with corners and $\sigma$ is a proper map.
\end{definition}
Not to confuse the two distinct notions of orientation we will say that cochains are \term{cooriented}.
Given $\sigma:P\rightarrow M$ and $\tau:P'\rightarrow M$ and a diffeomorphism $f:P\rightarrow P'$ such that $\sigma\circ f =\tau$, we get an induced isomorphism $det(\sigma)\rightarrow f^*(det(\tau))$.  Therefore, we can ask whether this map preserves orientations:  
\begin{definition}
A cochain $\sigma: P\rightarrow M$ is said to be \term{trivial} if there exists a diffeomorphism $f:P\rightarrow P$ with $\sigma\circ f=\sigma$ which changes the orientation of the map.
\end{definition}
The rest of the definitions go through as before and give rise to a chain complex with the boundary operator given by geometric boundary.  
\begin{definition}
Let $C^*(M)$ be the chain complex of proper maps. We denote its homology groups by $H^*(M)$.
\end{definition}
$C^*(M)$ is graded by the codimension of the map $\sigma:P \rightarrow M$.  With this convention, $\partial$ has degree 1.  As we will show later,  $H^*(M)$ is a geometric realization of the singular cohomology of $M$.\\\\
\textbf{Remark.}  In the case that $M$ is an infinite dimensional manifold, cochains are given by proper Fredholm maps $\sigma:P\rightarrow M$, where $P$ is necessarily infinite dimensional.  The rest of the construction proceeds with  little change.   

\section{Transversality and Pairing}
In this section we use standard results on transversality to obtain a pairing:
$$\cap: H_a(M)\otimes H^b(M)\rightarrow H_{a-b}(M)$$
The arguments are quite standard, but we need pay attention to the fact that perturbations should preserve trivial and degenerate chains. 

\begin{definition}
Maps $\sigma:P\rightarrow M$, $\tau:Q\rightarrow M$ are in \term{general position} if they satisfy the following conditions:  \\\\
1. They are transverse on each open stratum.  \\
2. If $P_i$ a component of an open stratum of $P$ with small image, then there exists a smooth map $g:T\rightarrow M$ of smaller dimension covering the image of $P_i$ which is transverse to all the strata of $\tau$. \\
3. If $Q_i$ a component of an open stratum of $Q$ with small image, then there exists a smooth map $g':T'\rightarrow M$ of smaller dimension covering the image of $Q_i$, which is transverse to all the strata of $\sigma$.   
\end{definition}
Now we exhibit enough perturbations to ensure transversality while preserving degenerate chains.  To this end, we have the following basic result:
\begin{lemma}
There exists a smooth connected manifold $\mathfrak{P}$ and a smooth map $$F: M\times \mathfrak{P} \rightarrow M$$ such that $D_2F$ is surjective at all points.  We assume $F(\cdot,p)=Id$ on $M$ for some $p\in \mathfrak{P}$.
\end{lemma}
Using standard transversality arguments we have the following:
\begin{lemma}
Given $\sigma:P \rightarrow M$ and $\tau: P\rightarrow M$ there exists $a\in \mathfrak{P}$ such that $F(\cdot,a)\circ \sigma$ and $\tau$ are in general position.  Given any two such $a,b\in \mathfrak{P}$, there exists a path $\gamma:[0,1]\rightarrow \mathfrak{P}$ such that $\gamma(0)=a$, $\gamma(1)=b$ and the map $$\Sigma:P\times [0,1]\rightarrow M$$ given by $\Sigma(p,t)=F(\sigma(p),\gamma(t))$ is transverse to $\tau$.
\end{lemma}

\begin{proof}
This is a standard transversality argument.  We need to ensure transversality for each open stratum as well as for some choice of maps $g_i:T_i\rightarrow M$ covering components with small image.  Since each chain has at most a countable number of components in each stratum, we can ensure simultaneous transversality by an application of Sard's theorem.  
\end{proof}
\begin{definition}
Given $\sigma$ and $\tau$ in general position, let $\sigma \cap \tau =(\sigma \times \tau) ^{-1}(\Delta)$ where $\Delta$ is the diagonal in $M\times M$.
\end{definition}
Note that $\sigma \cap \tau$ is oriented and comes with a map to $M$. Also, since $\sigma$ has compact domain and $\tau$ is proper, $\sigma \cap \tau$ is compact.
The need for such specific perturbations is to ensure taking intersections preserves degenerate and trivial chains:
\begin{lemma}
Assume $\sigma$ and $\tau$ are in general position.  If $\sigma$ is in $Q(M)$, so is $\sigma \cap \tau$.  
\end{lemma}
\begin{proof}
If $\sigma$ has a orientation reversing self-diffeomorphism $f$ such that $\sigma \circ f=\sigma$, then $$F(\cdot,a)\circ \sigma \circ f=F(\cdot,a)\circ \sigma $$  Therefore, $\sigma \cap \tau$ has an induced orientation reversing diffeomorphism as well. If $\sigma$ has small image contained in $g:T\rightarrow M$, then $F(\cdot,s)\circ \sigma $ has image contained in $F(\cdot,s)\circ g$ as long as $g$ is transverse to $\tau$. Thus, $\sigma \cap \tau$ has small image contained in $g \cap \tau$. Furthermore, if we assume $\sigma$ is degenerate, we have $\sigma \cap \tau$ has small image and $$\partial (\sigma \cap \tau)=\partial \sigma \cap \tau  \sqcup \pm \sigma \cap \partial \tau $$ is a union of s trivial chain and a chain with small image.      
\end{proof}

\begin{theorem}
Transverse intersections induce a well defined map: $$\cap: H_a(M)\otimes H^b(M)\rightarrow H_{a-b}(M)$$
\end{theorem}
\begin{proof}
Given a cycle $\sigma:P\rightarrow M$ and a cocycle $\tau:Q\rightarrow M$, we will use a generic homotopy $$F:M\times [0,1] \rightarrow M$$ to obtain a representative $F(\cdot, 1)\circ \sigma$ of the homology class of $\sigma$ that is transverse to $\tau$.  Suppose now that $\sigma \sim \partial \sigma_1$ where $\sigma$ is transverse to $\tau$.  If $\sigma_1$ is also transverse to $\tau$, then $\partial (\sigma_1 \cap \tau) \sim \sigma\cap \tau $ as desired.  If not, pick a generic homotopy $F$ as above and let $$\sigma_2(p,t)=F(\sigma_1(p),t)$$  Since $\partial(\sigma_1+\partial \sigma_2 )\sim \sigma$, we need only check that $\sigma_1+\partial \sigma_2 $ is equivalent to a transverse chain.  Let $$\sigma_3(p,t)=F(\sigma(p),t)$$  By construction,  $\sigma_1+\partial \sigma_2  \sim \sigma_3\sqcup \sigma_2(\cdot,1)$.  If we take a generic $F$, $\sigma_3$ and $\sigma_2(\cdot,1)$ will be transverse to $\tau$ as desired.  

\end{proof}

By a similar technique applied to proper maps , we obtain a geometric version of the cup product: $$H^a(M)\otimes H^b(M)\rightarrow H^{a+b}(M)$$

\section{Eilenberg-Steenrod Axioms}
In this section we  verify that our theory satisfies the Eilenberg-Steenrod axioms.  We begin by computing the homology of a point:
\begin{theorem}
$H_k(pt)=0$ when $n\neq 0$ and $H_0(pt)\cong \mathbb{Z}$.
\end{theorem}
\begin{proof}
Since any  map $\sigma:P\rightarrow pt$ is constant any such chain has small image when $dim(P)>0$.  If $\sigma$ is a cycle, we have $\partial \sigma$ a union of trivial and degenerate chains. Therefore, $\sigma$ is degenerate and hence $0$ in $C_*(pt)$.  To verify the claim in dimension zero note that $0$ dimensional cycles are finite collections of (oriented) points, while 1-dimensional cycles are compact intervals.
\end{proof}
Next, we verify a version of functoriality:
\begin{theorem}
Given a smooth map $f:M \rightarrow N$ between smooth manifolds, we have an induced chain map $f_*:C_*(M)\rightarrow C_*(N)$.  $f_*$ extends to a functor from the category of smooth finite dimensional manifolds and smooth maps to the category of chain complexes over $\mathbb{Z}$.
\end{theorem}
\begin{proof}
It suffices to note that the image of a trivial chain is trivial (with the same orientation reversing diffeomorphism) and given a map $\sigma$ with small image contained in $g:T\rightarrow M$, we have $f \circ \sigma$ has image contained in $f\circ g:T\rightarrow N$.
\end{proof}
\begin{theorem}
Given a smooth homotopy $F:M\times [0,1] \rightarrow N$, we have that $F_*(\cdot,0)$ is chain homotopic to $F_*(\cdot,1)$.
\end{theorem}
\begin{proof}
The chain homotopy $H$ is defined as follows.  Given a chain $\sigma:P\rightarrow M$ let $H(\sigma):P \times [0,1] \rightarrow M$ be the map with  $H(\sigma)(p,t)=F(\sigma(p),t)$.  We have $\partial H(\sigma)=(-1)^{|P|}(F(\sigma,1)-F(\sigma,0))+H(\partial \sigma)$.
\end{proof}
In the case $f$ is not smooth one can still define a map on the level of homology.  Indeed, one can approximate any continuous $f$ by a smooth $f'$.  Such a choice is unique up to smooth homotopy.  Therefore, the map on homology is still well defined as follows from the proof of the homotopy axiom.  \\\\
\textbf{Remark.}  The restriction that $M$ is smooth is not significant.  Since any finite CW complex $X$ has the homotopy type of a finite dimensional manifold (possibly noncompact), we can define our homology groups for such $X$ by choose a homotopy equivalent smooth model for $X$.\\\\
 Given an open set $U\subset M$ we have the exact sequence: $$0\rightarrow C_*(U)\rightarrow C_*(M)\rightarrow C_*(M)/C_*(U)\rightarrow 0 $$  We define $C_*(M,U)$ as $C_*(M)/C_*(U)$ and obtain the long exact sequence $$\cdot \rightarrow H_*(U)\rightarrow H_*(M)\rightarrow H_*(M,U) \rightarrow \cdots $$
Finally, we verify a version of the  excision axiom.  Let $U\subset M$ be an open subset and $K\subset U$ be compact set.  For excision, as well as later applications, we will need to cut cycles into pieces without changing the homology class.  Consider a smooth function $f:M\rightarrow \Ar$.   Let $$\sigma^+=\sigma_{|(f\circ \sigma)^{-1}[p,\infty)}$$ $$\sigma^-=\sigma_{|(f\circ \sigma)^{-1}(-\infty,p]}$$  and $$\sigma^0=\sigma_{|(f\circ \sigma)^{-1}(p)}$$
\begin{lemma}
Given $\sigma:P\rightarrow M$ with $\sigma \in Q(M)$, we have $\sigma^{\pm}$, $\sigma^0 \in Q(M)$ for generic $p$.  Let $\sigma:P\rightarrow M$ with $\partial \sigma \sim \tau$ where $\tau \in U\subset M$.  For generic $p$,  we have $[\sigma]=[\sigma^+\sqcup \sigma^-]$ in $H_*(M,U)$.  
\end{lemma}
\begin{proof}
If $\sigma$ is trivial so are $\sigma^{\pm}$ and $\sigma^0$ since the involution descends to them.  If $\sigma$ is degenerate, we have that $im(\sigma)\subset im(g)$ for $g:T\rightarrow M$. In addition, we have that $\partial \sigma =\alpha \sqcup \beta$ with $\beta$ trivial and $\alpha$ covered by $g'$ of smaller dimension.  Therefore, by choosing a generic $p$, we can assume its a regular value of $g'$, $g$ and $\sigma$.  
Now, $\sigma^{\pm}$ has small image since $\sigma$ has small image.  $\sigma^0$ has image contained in $g^0$ and hence has small image as well.  We have $$\partial \sigma^{\pm}=(\partial \sigma)^{\pm}\sqcup \pm \sigma^0$$ while $\partial \sigma^0=(\partial \sigma)^0$. Since $\alpha^\pm$ has small image and $\beta^\pm$ is we see that $\sigma^{\pm}$ and $\sigma^0$ are degenerate as well. \\
We apply the crease construction to conclude that $$\partial \Crease(\sigma)=\sigma^+\sqcup \sigma^- \sqcup -\sigma\sqcup \Crease(\partial \sigma)$$
Since $\partial \sigma \sim \tau$ we have $\Crease(\partial \sigma) \sim \Crease(\tau)$.  Indeed, given a chain $\rho \in Q(M)$, we have $\Crease(\rho)\in Q(M)$. This is clear for trivial chains.  For degenerate chains, observe that $\Crease(\rho)$ has same image as $\rho$ and hence is small.  Also, $$\partial \Crease(\rho)=\Crease(\partial \rho)\sqcup -\rho \sqcup \rho^+ \sqcup \rho^-$$  This shows that $\partial \Crease(\rho)$ has small image as well.  Since $\Crease(\tau)\in C_*(U)$, we conclude that $[\sigma]=[\sigma^+\sqcup \sigma^-]$ in $H_*(M,U)$ as desired.
\end{proof}

We can now prove the excision axiom:
\begin{theorem}
The natural inclusion $i_*:H_*(M-K,U-K)\rightarrow H_*(M,U)$ is an isomorphism.  
\end{theorem}
We break the proof into two lemmas:  
\begin{lemma}
$i_*$ is a surjection.  
\end{lemma}
\begin{proof}
Given a cycle $\sigma$ in $H_*(M,U)$, we show that it is homologous to $\sigma'$ with $$im(\sigma') \subset M-K$$ Take a smooth function $f:M\rightarrow \Ar$ with $f=0$ on $K$ and $f>1$ outside $U$.   Let $p\in (0,1)$ be a regular value of $f\circ \sigma$. Cutting by the preimage of $p$  decomposes $\sigma$ as $\sigma^+ +\sigma^-$ with $im(\sigma^-)\subset U$.  Since, $[\sigma]=[\sigma^+ \sqcup \sigma^-]$ in $H_*(M,U)$ and $[\sigma^-]=0$ in $C_*(M,U)$, we have $[\sigma]=[\sigma^+]$ in $ H_*(M,U)$.
\end{proof}
\begin{lemma}
$i_*$ is an injection.  
\end{lemma}
\begin{proof}
Consider a cycle in $H_*(M-K,U-K)$ represented by a chain $\sigma$.  Suppose $[\sigma]=0$ in $H_*(M,U)$.  Thus, there exists $\tau$ and $\eta\in C_*(U)$ with $\partial \tau \sqcup -\sigma \sim \eta$.  Since the image of $\sigma$ avoids $K$,  we can cut our chains by a generic hyperplane as above such that $\sigma^+=\sigma$ and $\sigma^-=\emptyset$.  Since cutting by a generic plane preserves $Q(M)$, we get $$(\partial \tau)^+ \sqcup -\sigma \sim  \eta^+$$ We have $\partial(\tau^+)=(\partial \tau)^+\sqcup \tau^0$.  Therefore, $$\partial (\tau^+) \sqcup -\sigma \sqcup -\tau^0\sim \eta^+$$ Since each chain in this union lies in $M-K$,  $[\sigma]=0$ in $H_*(M-K,U-K)$ as well. 
\end{proof}

Much of the previous discussion applies to the cohomology groups. For instance, given a smooth map $f:M\rightarrow N$ of manifolds without boundary, we get an induced map: $$f^*:H^*(N)\rightarrow H^*(M)$$ defined by pullback.  Note that it is only well defined on the homology level since we may need to perturb a cocycle representative to ensure that it is in general position. In the special case of an  inclusion of an open set, $i:U\subset M$, the map $i^*$ is well defined on the chain level since transversality is automatic.  We define $H^*(M,U)$ to be the homology of the kernel of $i^*$.

\section{Thom Isomorphism}
In this section we prove the Thom isomorphism theorem.  As we will show, this result is a direct consequence of the definitions and does not require an inductive Mayer-Vietoris argument.  Let $M$ be a smooth connected $m$-manifold without boundary and $V\rightarrow M$ a real vector bundle over $M$ of dimension $n$.  Let $\pi:V\rightarrow M$ be the projection and $i:M\rightarrow V$ be the zero section.
\begin{theorem}
$\pi^*:H^*(M)\rightarrow H^*(V)$ is an isomorphism with inverse $i^*$.  Similarly, the map $i_*:H_*(M)\rightarrow H_*(V)$ is an isomorphism with inverse $\pi_*$.  
\end{theorem}
\begin{proof}
Follows immediately from the fact that both $\pi\circ i$ and $i\circ \pi$ are homotopic to identity.
\end{proof}
Since we defined cohomology groups independently of homology, the fact that $i_*$ is an isomorphism does not immediately imply the same for $i^*$.  Let us restate the two results in geometric terms.  The isomorphism for $i_*$ simply asserts that any homology class in $V$ may by pushed down to lie on the zero section.  On the other hand, $i^*$ identifies cocycles of different dimension.  In essence, it asserts that a cohomology class is completely determined by its restriction to the zero section.  \\\\   
 Assume that $M$ is compact and that $V$ is oriented.  Note, that this implies that both $\pi$ and $i$ are oriented maps.  Since $i$ is proper, we get a map induced by inclusion:
$$i_*: H^*(M)\rightarrow H^{*+n}(V,V-i(M))$$
Recall that $H^{*+n}(V,V-i(M))$ is defined by considering the subcomplex of $C^*(V)$ of chains that vanish on $V-i(M)$. 
We have the following version of the Thom Isomorphism:
\begin{theorem}
$i_*: H^*(M)\cong  H^{*+n}(V,V-i(M))$ with inverse induced by $i^*$.
\end{theorem}

\begin{proof}
The proof is very similar to the proof of excision.  To show $i_*$ is surjective consider $[\sigma]\in H^*(V,V-i(M))$.  Introduce a metric on $V$.  We may cut $\sigma$ by the unit sphere bundle to get $[\sigma]=[\sigma^+\sqcup \sigma^-]$ with $\sigma^+$ supported away from the zero section.  By hypothesis, $\sigma^+\sim 0$.  Therefore, we may replace $\sigma$ by a chain contained in the unit disk bundle.  Projecting $\sigma$ to the zero section shows that $\sigma$ is homologous to a chain in $im(i_*)$.  Now, we establish surjectivity.  Suppose $i_*(\sigma)\sim \partial \tau$ in $C^*(V,V-i(V))$.  Again, we may cut $\tau$ by the unit sphere bundle to produce $\tau^-$ in the unit disk bundle of $V$ with $\partial \tau^- \sim \sigma$.  Projecting to the zero section yields $\tau'$ with $\partial \tau' \sim \sigma$.   
\end{proof}

We end this section by pointing out the relationship between our proof and the more standard proofs based on the notion of a Thom class. When the bundle is oriented, $$i:M\rightarrow V$$ may be viewed as a cohomology class $[U]\in H^n(V,V-i(V))$.  The restriction of $[U]$ is any fiber $V_p$ is the origin and hence the generator of $H^n(V_p,V_p-V_p^*)$.  This is the characterization of the Thom class.  The isomorphism  $$i_*:H^*(M)\rightarrow H^{*+n}(V,V-i(M))$$ may be viewed as the composition of pulling back by $$\pi^*:H^*(M)\rightarrow H^*(V)$$ followed by intersecting (the geometric analogue of the cup product) with the class $[U]$.   Note that the pullback, $i^*([U])\in H^n(M)$ is the Euler class associated to the bundle $V$.  This corresponds the intersection of the zero section with a generic section.\\\\
One can also recast the Thom isomorphism is a slightly different form.  Define $H^*_c(M)$ to be the group generated by oriented cochains which have compact domain.  If $V$ is oriented, we have $$i_*:H^*(M) \cong H_c^{*+n}(V)$$ induced by inclusion.

\section{Morse Homology Theorem}
In this section we prove the Morse homology theorem identifying the homology groups of the complex we constructed with those of the Morse complex.  There are many versions of this theorem in the literature.  We decided to include a proof since it has a direct generalization to the infinite dimensional theory.

\subsection{Passing a Critical Point}
Let $g$ be a Riemannian metric and $$f:M\rightarrow \Ar$$ be a proper Morse function, i.e. all critical points of $f$ are nondegenerate.  We assume that $f$ is bounded below. Choose a metric on $M$ such that $\nabla f$ is complete. We also assume that there is at most one critical point at each level, although this is not essential for the arguments that follow. Given a flow line $\gamma:[a,b]\rightarrow M$ with $\gamma'(t)=-\nabla f$ we define the energy of $\gamma$ as $$\int_a^b|\gamma'(t)|^2dt=f(a)-f(b)$$ Let $M^c=f^{-1}((-\infty,c))$. We let $\GF_t(M)$ be the time $t$ downward gradient flow. Thus, $$\GF_t \circ \GF_\tau =\GF_{t+\tau}$$  The flow acts on chains by $\GF_t(\sigma)=\GF_t \circ \sigma$. \\

First we show that only the presence of critical points can change the topology:
\begin{lemma}
Given $a<b$.  Assume there are no critical values of $f$ in $[a,b]$. Let $i:M^a\rightarrow M^b$ be the inclusion. We have the isomorphism $$i_*:H_*(M^a)\rightarrow H_*(M^b)$$
\end{lemma}
\begin{proof}
Given a cycle $\sigma \in H_*(M^b)$, $\GF_t(\sigma)$ lies in $M^a$ for sufficiently large $t$.  Since $\GF_t(\sigma)$ is homologous to $\sigma$, we have that $i_*$ is surjective.  Now, assume $i_*(\sigma)=0$.  Thus, $\partial \tau \sim \sigma$ with $\tau \in H_*(M^b)$. For large $t$, we have $\GF_t(\tau)\in C_*(M^a)$ and $\partial \GF_t(\tau)\sim \GF_t(\sigma)$.  Since the gradient flow preserves the homology class, we have $[\sigma]=0$ in $H_*(M^a)$ as well.   
\end{proof}
An easy compactness argument shows that:
\begin{lemma}
Assume the $f$ has a unique isolated critical point $y$ at level $c$. Choose $c'>c$ so close that $M^{c'}$ has no new critical points.  Let $U$ be a small neighborhood of $y$.  There exists $\ep>0$ such that any chain $\sigma:P\rightarrow M^{c'}$ may be pushed by the flow down to $\sigma_0$ with $im(\sigma_0)\subset U\cup M^{c-\ep}$.
\end{lemma}

Let $V$ be a finite dimensional vector space with  $dim(V)=dim(M)$.  The Morse lemma allows us to identify a neighborhood $U$ of $y$ with a neighborhood of the origin in $V=V^+\oplus V^-$ and $$f:V\rightarrow \mathbb{R}$$ with $$f(v)=-|v^-|^2+|v^+|^2+f(y)$$ Let $$ind(y)=dim(V^-)$$ By taking $\ep$ to be sufficiently small in the previous lemma, we can assume that $U$ contains the $2\sqrt{\ep}$ ball in $V$.  
\begin{theorem}
Let $k=ind(y)$.  For $\ep>0$ small, $H_k(M^{c+\ep},M^{c-\ep})=\mathbb{Z}$ and $H_j(M^{c+\ep},M^{c-\ep})=0$ for $j\neq k$.
\end{theorem}
\begin{proof}
Let $D_{2\sqrt{\ep}}$ denote the closed disk centered at $0$ of radius $2\sqrt{\ep}$ in $V$ and let $D^\pm(2\sqrt{\ep})$ denote the closed disk in $V^\pm$. By flowing $\sigma$ using the downward gradient flow, we may assume that the image of $\sigma$ is contained in $M^{c-\ep}\cup D_{\sqrt{\ep}}$.  By a small perturbation, we may also assume that $\sigma$ is transverse to $D^+(2\sqrt{\ep})$.\\\\
Our next task is to flatten out $\sigma$.  Pick a smooth function $$\phi:V \rightarrow [0,1]$$  Let $\phi$ be 1 on $D_{\sqrt{\ep}}$ and 0 outside $D_{2\sqrt{\ep}}$.   Consider the map $$H_t:V\rightarrow V$$ with  $$H_t(v^-,v^+)=(v^-,(1-t\phi(v))v^+)$$ and $t\in [0,1]$.  $H_0=Id$ while $H_1$ maps $V^+$ to 0 on $D_{\sqrt{\ep}}$.  $H_t\circ \sigma$ provides the desired homotopy that flattens out $\sigma$ near the origin.  Let  $\sigma_0=H_1\circ \sigma$.  Note that $$f(\sigma_0(p))\leq f(\sigma(p))$$  Thus, we can assume that points in $P$ that don't land in $M^{c-\ep}$ map to $V^-$ in the local coordinates.  Also note that $0$ is a regular value of this map since $\sigma$ was transverse to $D^+(2\sqrt{\ep})$. \\\\
We may cut our cycle $\sigma_0$ into two pieces $\sigma_0^- \sqcup \sigma_0^+$ by taking the preimage of a sufficiently small sphere $S^-(\delta)$ of radius $\delta < \sqrt{\ep}$ around the origin in $V^-$.  Here, we take  $\sigma_0^+$ to be the piece contained inside the ball $D^-(\delta)$ bounded by the sphere.  By using the crease construction, we have $[\sigma]=[\sigma_0^-\sqcup \sigma_0^+]$ in $H_j(M^{c+\ep},M^{c-\ep})$. Now  argument divides into cases depending on the dimension of the cycle. \\

\noindent Case $j<k$:  We have that $\sigma_0^+$ is empty and thus $\sigma_0=\sigma_0^-$ is contained in the region where $f<c$. It may be pushed to lie in $M^{c-\ep}$ by the flow.  Therefore, $H_j(M^{c+\ep},M^{c-\ep})=0$.\\ \\
Case $j=k$:  Here $\sigma_0^+$ is a finite collection of disks mapping diffeomorphically to $D^-(\delta)$.  Note that for $j=k$, $D^-(2\sqrt{\ep})$ defines an element in $H_k(M^{c+\ep},M^{c-\ep})$.  By replacing $\sigma$ with $\sigma\sqcup m[D_\ep^-]$, we can assume that $\sigma_0^+ \sim 0$ in $C_*(M^{c+\ep})$.  Therefore, $[\sigma]+m[D_\ep^-]$ is equivalent to a cycle strictly below $c$ and hence 0 in homology.  Thus, $D^-(2\sqrt{\ep})$ generates the homology in this dimension.  It's a free generator in $H_k(M^{c+\ep},M^{c-\ep})$ since its intersection with $D^+(\sqrt{2\ep})$ is 1.\\ \\
Case $j>k$:  In this final case, we note that since $\sigma_0^+$ maps to $D^-(\delta)$ it has small image as it is contained in a manifold of smaller dimension.  We have $$\partial \sigma_0^+=(\partial\sigma_0)^+\sqcup \sigma_0^0$$  Since $\partial\sigma_0 \sim 0$ we have $(\partial\sigma_0)^+ \sim 0$.  Also, $\sigma_0^0$ maps to $S^-(\delta)$ and hence has small image.  Therefore, $\sigma_0^+$ is degenerate. Thus, $\sigma$ is homologous to a cycle with image below the critical point and is therefore nullhomologous in $H_j(M^{c+\ep},M^{c-\ep})$.  

\end{proof}

More generally, we may assume that a critical level has $k$ critical points all of the same index.  In that case, $$H_*(M^{c+\ep},M^{c-\ep})\cong \mathbb{Z}^k$$  This proof is along the same lines as the one above.

\subsection{Self-Indexing Case}

In the previous section, we saw that if there is a unique nondegenerate critical point $y$ at $f(y)=c$ then
$H_{j}(M^{c+\ep},M^{c-\ep})=\mathbb{Z}$ when $j=ind(y)$ and $0$ else.  Suppose that the next
  critical point $z$ above $y$ is at $f(z)=d$ with $$ind(z)=ind(y)+1$$  Consider the composite map:
$$H_{j+1}(M^{d+\ep},M^{d-\ep})\xrightarrow{\partial_*}H_{j)}(M^{d-\ep})\xrightarrow{i_*}
 H_j(M^{d-\ep},M^{c-\ep})$$
Since we assumed there are no other critical points between $y$ and $z$, we have: $$H_*(M^{d-\ep},M^{c-\ep})\cong H_*(M^{c+\ep},M^{c-\ep})\cong \mathbb{Z}$$ 
 We wish to compute $i_* \circ \partial_*$. Since $H(M^{d+\ep},M^{d-\ep})$ is generated by the unstable disk $[D^-_z(\ep)]$,
  $$i_* \circ \partial_*([D^-_z(\ep)])=[\partial D^-_z(\ep)]$$ viewed as an element of $H_*(M^{d-\ep},M^{c-\ep})$.
  Using the gradient flow, we may push this cycle down to $M^{c+\ep}$.  As discussed above, the image
  in $H_*(M^{c+\ep},M^{c-\ep})$ is computed by counting intersections of the cycle with the disk $D^+_y(\ep)$ in
  a neighborhood of $d$.  This a version of the familiar statement that boundary operator may be computed by counting the intersections of attaching/belt spheres. \\\\ 
   Now, assume that $M$ is compact and $f$ is self indexing.  Thus, the critical points of index $i$ are on level $i$ of $f$. 
 We have a filtration $$\emptyset = M^{-1+\ep}\subset M^{0+\ep}\subset \cdots M^{n-1+\ep}\subset M^{n+\ep}=M$$
 with $H(M^{j+\ep},M^{j-1+\ep})=\mathbb{Z}^k$, where $k$ is the number of critical points of index $j$.
 
 \begin{definition}
 Let $C_*^f(M)=(\oplus_j C_j^f,\partial^f)$ be the free chain complex with $$C_{j}^f= H_*(M^{j+\ep},M^{j-1+\ep})$$
 The differential $\partial^f:C^f_*\rightarrow C^f_{*-1}$, $\partial^f=i_* \circ \partial_*$  arises from  composing
 $$\partial_*: H_{j}(M^{j+\ep},M^{j-1+\ep}) \rightarrow H_{j-1}(M^{k+1}) $$

  with
 $$i_*:  H_{j-1}(M^{j-1+\ep})\rightarrow H_{j-1}(M^{j-1+\ep},M^{j-2+\ep})$$
 Let $H_*^f(M)$ be the homology of the chain complex $C_*^f(M)$.
 \end{definition}

 \begin{theorem}
  Given a self-indexing Morse function $f$ as above, we have $H_*^f(M)\cong H_*(M)$
\end{theorem}
\begin{proof}
Given the local computation, the remaining algebraic argument is identical to the proof of the cellular homology theorem (see for example \cite{Hatch}).
\end{proof}
Let us summarize what we have accomplished.  Our only assumption on $f$ is that it is Morse and not necessarily Morse-Smale.  From this we deduced the existence of a chain complex generated by the critical points whose homology is isomorphic to the geometric singular homology of $M$.  The arguments, as well as those in the next section, make no use of gluing theory for gradient flow trajectories that is crucial in the Floer construction of the Morse complex.  Furthermore, when applied to a suitable infinite dimensional variant of the theory, they give an alternative construction of the Floer chain complex.   

\subsection{General Case}
The results of the previous section are not sufficient for applications in Floer theory.  The main issue is
that the assumption that the Morse function is self-indexing is too restrictive.  Indeed, the perturbations in
Floer theory are usually taken to be small and it is not clear to us how to perturb the relevant functional to obtain a self-indexing
one.  In this section we explain  how given that $\nabla f$ is Morse-Smale, we can define  the Morse chain
complex and show that its homology coincides with the geometric cycle homology. \\\\
We will make use of the following basic compactness result (see \cite{Jost}):
\begin{theorem}
Let $x$ and $y$ be critical points of $f$.\\\\
Given a sequence of possibly broken trajectories $$\gamma_i:[a_i, b_i] \rightarrow M$$ with $\gamma_i(a_i)\rightarrow x$ and $\gamma_i(b_i)\rightarrow y$, some subsequence converges to a possibly broken trajectory between $x$ and $y$.   
\end{theorem}
Assume we are given a proper Morse function $f$ which is Morse-Smale with respect to some metric on $M$.  The Morse-Smale condition implies that given critical points $x$ and $y$ with $ind(x) \leq ind(y)$, there are no gradient flow lines from $x$ to $y$.  In fact, there exist open neighborhoods $x\in U_x$  and $y\in U_y$ such that there are no gradient flow lines from $U_x$ to $U_y$.  Indeed, arguing by contradiction we can shrink the sets $U_x$, $U_y$ to obtain in the limit a possibly broken gradient flow line from $x$ to $y$.  
\begin{definition}
 A chain $\sigma$ is said to be $k-small$ if each critical point of index $j\geq k$ has a neighborhood $U$ such that
$U\cap \GF_t(\sigma)=\emptyset$ for all $t\geq 0$.
\end{definition}

\begin{definition}
 Let $C^k_*\subset C_*(M)$ be the subcomplex generated by $k$-small chains.  We have the filtration:
 $$ C_*^{0}\subset C_*^{1}\subset \dots C_*^{dim(M+1)}=C_*(M)$$
\end{definition}
\begin{lemma}
 $H_j(C^k_*)=0$ when $j>k$.
\end{lemma}
\begin{proof}
 Given a cycle $$\sigma: P\rightarrow M$$ in $C^k_j$ let us see what happens as we try to push it down by the flow.
 As we attempt to push it past a critical point $x$, two things can occur.  If $ind(x)>k$, $\GF_t(\sigma)$ stays away from an entire
 neighborhood of the critical point.  Therefore, we focus on flowing $\sigma$ past a critical point with $ind(x)\leq k<j$.   We will perturb $\sigma$ to $\sigma'$ in a small neighborhood of $x$ to be transverse to the stable manifold of $x$.  Here, we use the local stable manifold $V^+$ of the local model $V^+$ provided by the Morse lemma. A crucial point is that this perturbation does not take us out of $C^k_*$.  Indeed, since the flow is Morse-Smale, a small neighborhood of $x$ contains no points flowing arbitrary close the critical points of index $\geq k$. Therefore the homotopy $$\Sigma:P\times [0,1]\rightarrow M$$ perturbing $\sigma$ has image different from $\sigma$ only in a small neighborhood of $x$.   We have $j>k$, therefore $\sigma'$ intersects the stable manifold in a manifold of dimension at least 1.  By the arguments in the section on passing
a critical points, we can modify $\sigma'$ to $\sigma''$ locally around $x$ so that $\sigma''$ lies below $x$ and $[\sigma'']=[\sigma]$ in $H_*(C^k)$.  By repeating this a finite number of steps, we can modify $\sigma$ to be empty and thus $0$ in $H_j(C^k_*)$.
\end{proof}

\begin{lemma}
  $H_j(C^k_*,C^{k-1}_*)=0$ for  $j\neq k$.
\end{lemma}
\begin{proof}
 The argument is quite similar to the proof of the previous lemma.  We are given a chain $\sigma \in C^k_j$ and 
 $\partial \sigma \in C^{k-1}_{j-1}$.   As before, we are trying to push our chain
 past a critical point $x$.  We need to check that all the constructions used in the lemma regarding passing a critical point carry over to this setting.  First we need perturb the chain in a neighborhood of $x$ to be
transverse to the stable manifold.  The perturbation gives a chain $$\Sigma:P\times [0,1]\rightarrow M$$ with
$\Sigma_{|P\times 1}=\sigma$ and  $$\partial \Sigma=\Sigma_{|P\times 1}\sqcup \Sigma_{P\times 0}\sqcup \Sigma_{\partial(P)\times[0,1]}$$
  Note that $\Sigma_{\partial(P)\times[0,1]} \in C^{k+1}_*$ since the perturbation is supported
 in a neighborhood of $x$ and thus is zero in $C^k_*/C^{k+1}_*$.  Therefore, we can perturb the cycle to be
transverse without changing its homology class.  Next, we project the cycle to manifold $V^-$.  This does not change our class.   If $j<ind(x)$ the intersection with $V^+$ is empty hence we can flow past it.  
 Otherwise we cut our $\sigma$ to produce $\sigma^+\sqcup \sigma^-$ with
 $\sigma^-$ strictly below the critical point.  Since $j> ind(x)$, $\sigma^+$ is degenerate and thus can be discarded.  Therefore we can flow
past any critical point $x$ without altering the homology class.  Eventually, the class is empty. 
\end{proof}
Take $D^{\pm}(x_m)$ to be a  disk in $V^{\pm}(x_m)$ around each critical point 
$x_m$ of index $k$ small enough to be in $U_{x_m}$.  Note that $D^-(x_m)$ is a cycle in $C^k_{k}/C^{k-1}_{k}$.
Indeed, the boundary of such a disk can flow only to a critical point of index $\leq k-1$ and thus contained in $C^k_{k-1}$.  Our next task is to verify that these disks freely generate $H_{k}(C^k_*,C^{k-1}_*)\cong \mathbb{Z}^n$.  We say that chain $\sigma$ lies below $S\subset M$ if $\GF_t(\sigma)$ is disjoint from $S$ for all $t \geq 0$.
\begin{lemma}
Any chain in $C_*^k$ can be pushed down by the flow to lie below $\cup_m \partial D^+(x_m)$
\end{lemma}
\begin{proof}
Consider $\GF_t(\sigma)$ for $t$ very large.  We claim that $$im(\GF_t(\sigma))\cap (\cup_m\partial D^+(x_m))=\emptyset$$ for all $t$ sufficiently large.  By contradiction, assume there are large times $t_i$ and trajectories starting at points of $\sigma$ and ending on some $\partial D_m$.  By compactness such trajectories must approach some critical point arbitrarily closely somewhere along the way.  Let $y$ be such a critical point.  Note that $ind(y)\leq k$ since $\sigma\in C^k_*(M)$.  We claim that $ind(y)<k$.  First, if $y=x_m$, the value of $f$ along the trajectory would lie strictly below $\partial D_m^+$.  If $y\neq x_m$, part of the trajectory represents a flow line between small neighborhoods of critical points of the same index.  This is not possible since $D^+(x_m)$ is chosen to lie in $U_{x_m}$.
\end{proof}

\begin{lemma}
$H_{k}(C^k_*,C^{k-1}_*)\cong \mathbb{Z}^n$, where $n$ is the number of critical points of index $k$.
\end{lemma}
\begin{proof}
We define a  map $$Ev:H_{k}(C^k_*,C^{k+1}_*)\rightarrow \mathbb{Z}^n$$
 which will turn out to be be an isomorphism.  By the previous lemma, we can push any $\sigma\in H_{k}(C^k_*,C^{k+1}_*)$ to lie below $\cup_m \partial D^+(x_m)$.  We may perturb such $\sigma$ to be transverse to each $D^+(x_m)$ by altering it near $x_m$.  Let $$Ev(\sigma)=(a_{x_1},a_{x_2},\cdots )$$ be the intersections of $\sigma$ with $D(x_i)$.  Note that this does not depend on the small perturbation since $\partial \sigma$ lies in $C_*^{k-1}$ and thus avoids $U_{x_m}$.  We now assert that this gives rise to a
well defined intersection with the $D^+(x_m)$.  To see this, take $\sigma$ below $\cup_m\partial D^+(x_m)$ and assume $\sigma=\partial \tau$ with $\tau \in C_*^k$.  Taking $T$ sufficiently large $\GF_T(\tau)$ lies below $\cup_m\partial D^+(x_m)$ and $\partial \GF_T(\tau)=\GF_T(\sigma)$.  However, $\sigma$ and $\GF_T(\sigma)$ are homotopic on the complement of $\cup_m\partial D^+(x_m)$.  Therefore, they have the same intersection with $\cup_mD^+(x_m)$.    Thus, $Ev$ vanishes on boundaries and is therefore well defined.  $Ev$ is surjective since $D^-(x_a)\cap D^+({x_b})=\delta_{ab}$.  $Ev$ is also injective since if $\sigma$ has zero algebraic intersection with all $D^+(x_m)$ we can modify it to lie below all critical points of index $k$, as in the previous lemmas.
\end{proof}

With these results in place, we can proceed to define the Morse complex as  $$C^f_i(M)=H_i(C^{i}_*,C^{i-1}_*)$$
The differential $\partial^f:C^f_*\rightarrow C^f_{*-1}$, $\partial^f=\partial_* \circ i_*$  arises from the connecting homomorphism:
 $$\partial_*: H_{k}(C^{k}_*,C^{k-1}_*) \rightarrow H_{k-1}(C^{k-1}_*) $$
and the one induced by inclusion:
 $$i_*:  H_{k-1}(C^{k-1}_*)\rightarrow H_{k-1}(C^{k-1}_*,C^{k-2}_*)$$
Geometrically, the differential may be interpreted as follows.  The generators for $H_k(C^k_*,C^{k-1}_*)$ can be taken to be small oriented disks $D^-_{x_m}$ in the unstable manifolds around the critical points as above.
Let $D^+_{y_l}$ be the stable cooriented disks around critical points of index $k-1$.  Take the intersections
$\GF_T(\partial D^-(x_m))\cap D^+(y_l)$ for $T$ large.   These give you the coefficients of the incidence matrix.
By the same arguments as in \cite{Hatch}, we have the following Morse Homology Theorem:
\begin{theorem}
There exists an isomorphism $$H_*(M)\cong H^f_*(M)$$
\end{theorem}
Finally, we would like to remark that there is a version of the Morse homology theorem for cohomology.  In this case, one considers the upward gradient flow which exists as long as $\nabla f$ is complete.  We can define $H_f^*(M)$ much like before.  We have:
\begin{theorem}
There exists an isomorphism $$H^*(M)\cong H_f^*(M)$$
\end{theorem}

\section{Equivariant Theory}
\subsection{Basic Construction}
The theory of geometric cycles has an equivariant generalization.  Below, we describe a construction which can be viewed as the geometric analogue of the Cartan construction for deRham theory.  A similar theory for singular homology appears in \cite{Jones}.  For this section, assume $M$ has a smooth action $$S^1\times M \rightarrow M$$  
\begin{definition}
Let $\twist$ be the map taking $\sigma:P\rightarrow M$ to $$\twist(\sigma):S^1 \times P \rightarrow  M$$ with $$\twist(\sigma)(e^{i\theta},p)=e^{i\theta}\sigma(p)$$
\end{definition}
We define our complex as follows. Let $$\Cep_*=(C_*(M)\otimes \mathbb{Z}[u],\partial_\twist)$$ with $u$ a formal variable of degree $-2$ and $$\partial_\twist(\sigma u^k)=(\partial \sigma) u^k+\twist(\sigma) u^{k+1}$$  The grading of $\sigma u^k$ is $dim (\sigma)-2k$.  
\begin{lemma}
$\partial_\twist^2=0$.  
\end{lemma}  
\begin{proof}
We compute $$\partial_{\twist}^2\sigma=\partial_{\twist}(\partial(\sigma)+\twist(\sigma)u)=\partial^2(\sigma)+\twist(\partial \sigma)u+\partial(\twist \sigma)u+\twist^2(\sigma)u^2$$
Observe that since $\twist(P)$ is already $S^1$ invariant $$\image \twist^2\sigma=\image (\twist\sigma)$$ and $$\image(\partial \twist^2\sigma)=\twist(\partial \twist(\sigma))$$  By definition, this implies that $\twist^2(P)$ is degenerate and hence zero in $C_*(M)$.  $$\twist(\partial \sigma):S^1\times \partial P\rightarrow M$$ while $\partial(\twist \sigma):\partial(S^1\times P)\rightarrow M$.  Therefore, $\twist(\partial \sigma)=-\partial(\twist \sigma)$.
\end{proof}

\begin{definition}
Let $\Hep_*(M)$ be the homology associated with the complex $\Cep_*(M)$.
\end{definition}
We have the following analogue of the Gysin sequence:
\begin{theorem}
There is a long exact sequence:
$$\dots \rightarrow \Hep_{k+2}(M)\rightarrow \Hep_k(M) \rightarrow H_k(M)\rightarrow \dots$$
\end{theorem}
\begin{proof}
We have an exact sequence of complexes: $$0\rightarrow \Cep_{*+2}(M) \xrightarrow{u} \Cep_{*}(M) \rightarrow C_*(M)\rightarrow 0$$
\end{proof}
\begin{theorem}
Given a smooth $S^1$-equivariant map $f:M\rightarrow N$, we have an induced map $$f_*:\Hep_*(M)\rightarrow \Hep_*(N)$$  

\end{theorem}
\begin{proof}
We define $f_*(\sigma u^k)$ as $(f\circ \sigma) u^k$.  If $\sigma:P\rightarrow M$, we have $$\twist(f\circ \sigma):S^1\times P \rightarrow N$$ with $$\twist(f\circ \sigma)(e^{i\theta},p)=e^{i\theta}f(\sigma(p))$$  Equivariance of $f$ implies that  $$f_*(\twist(\sigma))(e^{i\theta},p)=f(e^{i\theta}\sigma(p))=e^{i\theta}f(\sigma(p))$$  Therefore, $\twist(f_*)=f_*(\twist)$ as desired.  
\end{proof}
The homotopy axiom is also easy to verify:
\begin{theorem}
Given an equivariant map $H:M\times [0,1]\rightarrow N$, we have that $H_*(\cdot,0)=H_*(\cdot,1)$.
\end{theorem}
We now introduce two variants of the equivariant construction:
\begin{definition}
 Let $\Hei_*(M)$ be the homology associated with the complex $$\Cei_*(M)=(C_*(M)\otimes \mathbb{Z}[u,u^{-1}],\partial_\twist)$$  and $\Hem_*(M)$ be the homology of the complex $$\Cem_*(M)=C_*(M)\otimes \mathbb{Z}[u,u^{-1}]/u\cdot (C_*(M)\otimes \mathbb{Z}[u])$$
\end{definition}
Since localization is exact, we have:
\begin{theorem}
$\Hei_*(M)\cong u^{-1}\Hep_*(M)$
\end{theorem}
\begin{theorem}
 There is a long exact sequence: $$\dots \rightarrow \Hep_{k+2}(M)\rightarrow \Hei_k(M) \rightarrow \Hem_k(M)\rightarrow \dots$$
\end{theorem}
\begin{proof}
This follows from the short exact sequence of complexes:
$$0\rightarrow \Cep_{*+2}(M)\xrightarrow{u} \Cei_*(M) \rightarrow \Cem_*(M) \rightarrow 0$$
\end{proof}
\textbf{Remark.} One can show (see \cite{Jones2} for a discussion) that $\Hei_*(M)$ is the contribution to the homology coming from the fixed points of the action.  In particular, one has the localization formula: $$\Hei_*(M)\cong H_*(M^{fix})\otimes \mathbb{Z}[u,u^{-1}]$$ In the case $M$ is infinite dimensional, one must consider a completion of the groups with respect to $u$ for the localization formula to hold.  
\subsection{Cohomology and Pairing}
We can define cohomological versions of these groups that will turn out to be dual to the ones in the previous section.  
\begin{definition}
Let $v$ be a formal variable of degree 2.  Let $\HHep^*(M)$ (resp. $\HHem^*(M)$, $\HHei^*(M)$) be the group associated to the complex $(C^*(M)\otimes \mathbb{Z}[v],\partiale)$ (resp. $(C^*(M)\otimes \mathbb{Z}[v,v^{-1}],\partiale), (C^*(M)\otimes \mathbb{Z}[v,v^{-1}]/v(C^*(M)\otimes \mathbb{Z}[v]),\partiale))$. 
\end{definition}
Now, we discuss the cohomology pullback map which will use transversality theory.  We state the results for the ``$+$'' version, but they have straightforward generalizations to all the versions.
\begin{theorem}
Given an equivariant map $f:M\rightarrow N$, we have an induced map $f^*:\HHep^*(N)\rightarrow \HHep^*(M)$.  This construction is functorial.
\end{theorem}
\begin{proof}
Given a cocycle $\sigma_0+\sigma_1 v\cdots$ our first task is to perturb it to be transverse to $f$. We must not change the equivariant cohomology class during the perturbation.  Using usual transversality arguments, let $\partial \tau_0=\sigma_0-\sigma'_0$ with $\sigma'_0$ transverse to $f$.  Therefore, we may alter our class to be $$\sigma_0'+(\sigma_1-\twist \tau_0)v+\cdots$$  Now, perturb $\sigma_1-\twist \tau_0$ to be transverse to $f$.  This time the correction lies in the $v^2$ term.  We may continue in this fashion until the terms become degenerate.  Therefore, we may assume that $\sigma=\sigma_0+\sigma_1 v\cdots$ is transverse to $f$.  We set $$f^*(\sigma)=\sigma_0\cap f+\sigma_1 \cap f v\cdots$$
  We need to check that $f_*$ commutes with $\partiale$ on transverse chains.   It is clear that $$\partial(f^*(\sigma))=f^*(\partial \sigma)$$  If $$\sigma_i:P\rightarrow N$$ we have $$f^*(\sigma_i):Q\rightarrow M$$ where $Q=(p,m)$ with $f(m)=\sigma(p)$.  Therefore, $$\twist(f^*(\sigma_i)):S^1\times Q\rightarrow M$$    Meanwhile, $$f^*(\twist \sigma_i):Q'\rightarrow M$$ where $Q'=(e^{i\theta},p,m)$ with $f(m)=e^{i\theta}\sigma_i(p)$.  The diffeomorphism $$\phi:Q\rightarrow Q'$$ sending $(e^{i\theta},p,m)$ to $(e^{i\theta},p,e^{i\theta}m)$ provides the isomorphism of $\twist(f^*(\sigma_i))$ and $f^*(\twist(\sigma_i))$.  Therefore, $\partiale(f^*(\sigma))=f^*(\partiale \sigma)$ as desired.   
\end{proof}
By a similar pullback argument, we have the homotopy axiom:
\begin{theorem}
Given an equivariant map $H:M\times [0,1]\rightarrow N$, we have that $H^*(\cdot,0)=H^*(\cdot,1)$ 
\end{theorem}
While transversality arguments may fail in the equivariant context, the previous results give a way of providing using proofs based on general position by keeping track of the "error" terms.  This idea will be significant in the construction of Chern classes below. \\\\  One may use transversality to pair homology with cohomology:

\begin{lemma}
There is a pairing $\HHem^a(M)\otimes \Hep_a(M)\rightarrow \mathbb{Z}$
\end{lemma}
\begin{proof}
Let us say that $\sigma$ is \textbf{strongly transverse} to $\tau$ $\tau \pitchfork \sigma$ and $\tau \pitchfork \twist(\sigma)$.  Given chains $\Sigma_{i=0}^n\sigma_i v^i\in C_*^+(M)$ and $\Sigma_{i=-n}^0\tau_i u^i \in C_*^-(M)$ with $\sigma_i$ strongly transverse to $\tau_{-i}$ we define the intersection as $$\Sigma_{i=0}^n \#(\sigma_i\cap \tau_{-i})\in \mathbb{Z}$$  The calculation below is based on the following observation.  Given $\sigma:P\rightarrow M$ and $\tau:Q\rightarrow M$ where $\sigma$ has dimension $i$ and $\tau$ has codimension $i-1$,  we have  $$\#(J(\sigma)\cap \tau)=\#((-1)^{|\sigma|+1}\sigma \cap J(\tau)) $$  Indeed, we have the bijection $$(e^{i\theta},p,q)=(p,e^{-i\theta},q)$$ for triples with $e^{i\theta }\sigma(p)=\tau(q)$.  \\\\
We have $$\#(\partial_{\twist}(\sigma)\cap \tau +(-1)^{|\sigma|} \sigma \cap \partial_{\twist}\tau)$$ $$=\Sigma_{i=0}^n\#((\partial \sigma_i+\twist\sigma_{i-1})\cap \tau_{-i}+(-1)^{|\sigma|}\sigma_i \cap  (\partial \tau_{-i}+\twist \tau_{-i-1}))$$ $$=\Sigma_{i=0}^n\#((\partial \sigma_i)\cap \tau_{-i})+(-1)^{|\sigma|}\#(\sigma_i \cap  \partial \tau_{-i})=\Sigma_{i=0}^n\#(\partial (\sigma_i\cap \tau_{-i}))=0$$  This calculation implies that the intersection count is well defined. 
To complete the construction we must check that the chains may be perturbed to be transverse without changing the homology class. Assume inductively that $\sigma_0$ through $\sigma_{k-1}$ are strongly transverse to all $\tau_j$.  Let $$\rho:P_k\times [0,1] \rightarrow M$$ be a generic homotopy of $\sigma_k$ such that $\rho(\cdot, 0)=\sigma_k$ and $\rho(\cdot, 1)=\sigma_k'$ is strongly transverse to $\tau_{-j}$. Taking $\partial_J \rho$ we may replace $\sigma_k$ by $\sigma_k'+\rho'$ where $\rho'$ is a generic homotopy of $\partial \sigma_k$.   However, $\partial \sigma_k \sim J(\sigma_{k-1})$ since $\sigma$ is a cycle.  Since $\sigma_{k-1}$ is strongly transverse to all $\tau_j$, we may assume by a generic choice of $\rho$ that $\rho'$ equivalent on the level of chains to a strongly transverse representative.   After finitely many steps, we obtain a chain where the only non transverse chains have coefficient $u^l$ so large that $\tau_{-l}$ is empty.
\end{proof}
\noindent \textbf{Remark.} It is not immediate how to define a cup product on $\HHep^*(M)$ directly.  Indeed, on $M\times M$ with the diagonal $S^1$-action, the operator $J$ does not decompose as an operator on the factors.  Thus, given cohomology classes $a,b\in H^*_+(M)$, $a\times b$ does not define an equivariant class in $M\times M$.  In fact, the difference $$\{a,b\} := \twist (a\cap b)-\twist (a)\cap b-(-1)^{|a|}a\cap \twist (b)$$ is a finite dimensional analogue of the string bracket of Chas-Sullivan (see \cite{Sullivan}).    

\subsection{Some Calculations}
Here are a few simple calculations:
\begin{lemma}
$\Hep_*(pt)\cong \mathbb{Z}[u]$; $\Hei_*(pt)\cong \mathbb{Z}[u,u^{-1}]$; $\Hem_*(pt)\cong \mathbb{Z}[u,u^{-1}]/\mathbb{Z}[u]$
\end{lemma}
\begin{proof}
These follow directly from the fact that for a point, both $\twist$ and $\partial$ are zero.
\end{proof}
\begin{lemma}
$\Hep_*(S^1)\cong \mathbb{Z}$; $\Hei_*(S^1)\cong 0$; $\Hem_*(S^1)\cong \mathbb{Z}$
\end{lemma}
\begin{proof}
$\Hep(S^1)$: Let $\sigma=\sigma_0+\sigma_1u+\sigma_2u^2 \cdots$ be a cycle and let $$id:S^1\rightarrow S^1$$ be the identity map.\\

Suppose $\sigma$ is a cycle of degree $>1$.  Therefore,  $\sigma_i$ is automatically small.  We have $\partial \sigma_0 \sim 0$, and thus $\sigma_0$ is degenerate while  $\sigma_i$ is degenerate for $i>0$ simply because $dim(\sigma_i)>2$. 

Suppose $\sigma$ has degree 1.  Since $dim(\sigma_i)>2$ for $i>0$, $\sigma_i\sim 0$ for $i>0$.  Furthermore, $\partiale \tau=\partial \tau_0$.  Therefore, in degree 1, we are reduced to the ordinary homology of $S^1$ which has $id$ as a free generator.

Suppose $\sigma$ has degree 0.  $\sigma_i\sim 0$ for $i>1$ since all such chains have $dim>2$ and are therefore degenerate.  Thus, $\sigma=\sigma_0+\sigma_1u$ with $\twist(\sigma_0)\sim -\partial\sigma_1$.  $\sigma_0$ is a finite collection of oriented points. Therefore, $\twist \sigma_0$ is a finite number of copies of $i$.  By the computation of the ordinary homology of $S^1$,  $\twist(\sigma_0)\sim -\partial\sigma_1$ implies that the algebraic count of $\sigma_0$ is trivial.  Thus, there exists $\tau_0$ with $\partial \tau_0=\sigma_0$.  Therefore, we can assume that $\sigma=\sigma_1u$.  In this case, $\sigma_1$ is degenerate since it has small image and $\partial \sigma_1\sim 0$.
 
Suppose $\sigma$ has degree $<0$.  We can factor $\sigma=u^k(\sigma')$ with $\partiale \sigma'=0$ and degree $\sigma'$ either $0$ or $1$.  If degree $\sigma'$ is 0,  we have $\sigma'=\partiale \tau$ hence $\sigma=\partial (u^k\tau)$.  Now, assume that the degree of $\sigma'$ is 1.  If $k>1$ we have $\sigma=u^{k-1}(u\sigma')$.  Since $u\sigma'$ is a cycle of degree $>1$, we have $\partiale\tau=u\sigma'$ and thus $\partiale \sigma=\partiale (u\tau)$.  Finally, we deal with the case $\sigma=u\sigma'$ where $\sigma'$ has degree 1.  We have $[\sigma']=n[id]$ in $\Hep_*(M)$.  However, $\partiale(pt)=id[u]$ which implies that $\sigma$ is a boundary.      \\\\  
$\Hei(S^1)$:  By localization, $$\Hei(S^1)=u^{-1}\Hep(S^1)=0$$  Of course, a direct argument is also possible.\\\\
$\Hem(S^1)$:  This follows from the long exact sequence connecting the groups.  Note that in this case $[pt]$ is a cycle since  $$\partial_{\twist}[pt]=id[u]=0$$ and $u=0$.  $$id:S^1\rightarrow S^1$$ on the other hand is a boundary since $[S^1]=\partial_{\twist }[pt]u^{-1}$.\\

\end{proof}

\subsection{Relation to the Borel Construction}
Let us first recall the Borel construction of equivariant homology.  Let $S^\infty$ be the unit sphere in an infinite complex Hilbert space.  $S^\infty$ is contractible and has a smooth $S^1$ action.  We let $\C P^\infty$ be the quotient $S^\infty/S^1$.  Let $M$ be a smooth manifold with a smooth $S^1$-action. The Borel model of $M$ is defined as $$M_{S^1}=(S^\infty \times M)/S^1$$  Let $$H_*^B(M)=H_*(M_{S^1})$$  
In this section we demonstrate a natural  isomorphism $$H^-_*(M)\cong H_*^B(M) $$
\textbf{Remark.}  Instead of introducing infinite dimensional manifolds one may take a finite dimensional approximation $S^k$ to $S^\infty$ and define $H_*^B(M)$ as the limit $$\lim_{k\rightarrow \infty} H_*((M\times S^k)/S^1)$$
As a preliminary step, we have the following lemma:
\begin{lemma}
Assume that $M$ has a free $S^1$ action.  We have $ H_*(M/S^1) \cong H_*^-(M)$ with the map induced by the projection  $$\pi: M\rightarrow M/S^1$$    
\end{lemma}
\begin{proof}
Given a chain $\sigma=\sum_{i=0}^n \sigma_i u^{-i}\in H^-_*(M)$,  let $F_*(\sigma)=\sigma_0\in C_*(M/S^1)$.  To see that $F_*$ is a chain map note that $F_*(J(\sigma))$ is always degenerate and hence 0 in $C_*(M/S^1)$.  The rest of the argument is a standard application of the Mayer-Vietoris sequence.  Let $U\subset M/S^1$ be a small ball.  The restriction of $\pi$ to $U$ is equivalent to the projection $S^1\times U \rightarrow U$.  Since $$H_*^-(S^1\times U)\cong \mathbb{Z}$$ with the generator given by $p\times q$ for any $(p,q)\in U$, we have the isomorphism $H_*^-(S^1\times U) \cong H_*(U)$ induced by $F_*$.   We may cover $M/S^1$ by convex balls and proceed by induction to deduce the general case.  
\end{proof}

\begin{corollary}
The projection $\pi: M\times S^\infty \rightarrow M_{S^1}$ induces an isomorphism $$\pi_*: H^-_*(M\times S^\infty) \cong H_*(M_{S^1})$$  
\end{corollary}
The isomorphism $H^-_*(M)\cong H_*^B(M) $ now follows from the following lemma:
\begin{lemma}
The equivariant projection $\pi_1:M\times S^\infty \rightarrow M $ induces an isomorphism on homology.  
\end{lemma}
\begin{proof}
As $S^\infty$ is contractible, the result is a consequence of the following general fact.  Suppose  we have an equivariant map $f: M\rightarrow N$.   If  $$f_*:H_*(M)\rightarrow H_*(N)$$ is an isomorphism, then so is $$f_*:H_*^-(M)\rightarrow H_*^-(N)$$  The proof of this fact is the following filtration argument.  Let $C_*^-$ be the chain complex of either $M$ or $N$.  Let $F^kC_*^-\subset C^-_*$ be the subcomplex generated by sums $\sum_{i=0}^k \sigma_i u^{-i}$.  We have  $$F^kC_*^-/F^{k-1}C_*^- =C_{*-2k}$$ Therefore, since $C_a=0$ for $a <0$, in each degree $l$ the filtration $F^kC_l^-$ is stable for $k$ sufficiently large.  The conclusion follows from the 5-lemma by comparing the long exact sequences for the pair $(F^kC_*^-, F^{k-1}C_*^-)$.      
\end{proof}

\subsection{Chern Classes and Equivariant Cohomology}
Assume $$\pi: V\rightarrow M$$ is a complex vector bundle of dimension $n$.  In this section we describe a construction of Chern classes using equivariant cohomology.  In particular, transversality theory plays a major role in the construction.  Observe that $V$ has a natural $S^1$-action induced by the restricting the action of $\mathbb{C}^*$ to $S^1$ on each fiber.  We think of $M$ as having a trivial $S^1$-action. Let $$i:M\rightarrow V$$ be the zero section.   
\begin{theorem}
We have $i^*: \HHep^*(V) \cong \HHep^*(M)\cong H^*(M)\otimes \mathbb{Z}[v]$
\end{theorem}
\begin{proof}
Follows from the fact that $\pi \circ i=Id$, while $i\circ \pi$ is equivariantly homotopic to the identity.  
\end{proof}
 Recall that the Thom class $U$ is given by $i:M\rightarrow V$.  In this section, we view it as equivariant cohomology class in $\HHep^*(V)$.  
\begin{definition}
Let $Ch(V)=i^*(U)\in H^*_+(M)$ be the total Chern class.  
\end{definition}
Note that $Ch(V)$ decomposes as $$[ch_n(V)]+[ch_{n-1}(V)]v+[ch_{n-2}(V)]v^2\cdots $$ with $ch_{n-i}(V)$ of degree $2(n-i)$  Let us unravel this definition.  We need to modify $U$ to produce an equivariant class transverse to $i:M\rightarrow V$.  Take $\tau_0$ with $\partial \tau_0 \sim U-U'$ where $U'$ is transverse to $i$.  We have $$U-\partiale \tau_0=U'-(\twist\tau_0) v$$  Now, pick $\tau_1$ such that $\twist\tau_0-\partial \tau_1$ is transverse to $i$.  This time the correction term is $(\twist \tau_1 )v^2$.  Repeat until the correction term is degenerate and hence can be ignored.  We have produced a chain $$ U'+a_1 v+a_2 v^2+\cdots $$ such that each term is transverse to $i$.  We can pull this back to obtain $Ch(V)$ as desired.  In particular, this implies that $ch_n(V)$ is the Euler class of $V$.  \\\\
\textbf{Example.} Here is the simplest version of this construction.  Let us consider $\C \rightarrow pt$ as a line bundle over the origin.  We define the homotopy $$H:[0,1]\rightarrow \C$$ with $H(t)=t\subset \C$.  The resulting section is $1\subset \C$ that is now transverse to $0$ and its pullback is the empty set.  Thus, the first Chern class is 0.  According to our construction, we must examine the map $$\sigma: S^1\times [0,1]\rightarrow \C$$ given by mapping $(e^{i\theta},t)$ to $e^{i\theta}t$.  $\sigma$  does not have zero as a regular value. However, any other point in the open unit ball is a regular value of $\sigma$ and has intersection $1$ with the chain $\sigma$. A small perturbation of $\sigma$ produces a chain $\sigma'$ that has intersection 1 with the origin.   Therefore, the pullback of the zero section is $1\in H^*(pt)$.  A version of this argument applies generally to a complex vector bundle $V\rightarrow M$ to conclude that $ch(V)_0=1$.       

\begin{lemma}
We have $Ch(V_1\oplus V_2)=Ch(V_1)\cup Ch(V_2)$. 
\end{lemma}
\begin{proof}
Let $U_i$ be the equivariant cochain in $C_+^*(V_i)$ given by the zero section $$s_i:M\rightarrow V_i$$ Consider the vector bundle $V_1\oplus V_2$ over $M\times M$. The cochain $$U_1\times U_2  \in C_+^*(V_1\oplus V_2)$$ given by the zero section  $i:M\times M \rightarrow V_1\oplus V_1$ is the Thom class of $V_1\oplus V_2$.  Since, $$\pi_i^*:\HHep^*(M)\rightarrow \HHep^*(V_i)$$ is an isomorphism, we can take $\pi_i^*(Ch(V_i))=U'_i$ as a cocyle that also represents the Thom class $U_i$.  In other words, $$U_i-U'_i \sim \partial_J \tau_i$$    We claim that $U_1'\times U_2' $ represents the Thom class of $V_1\oplus V_2$.  This follows from the Leibniz rule for $\partial_J$.  Indeed, one must check that $J(\sigma\times \tau)=J(\sigma)\times \tau$ in the special case where $\tau$ has an $S^1$-action.  If $\sigma:P\rightarrow V_1$ and $\tau:Q \rightarrow V_2$, the map sending $(e^{i\theta}, p, q)$ to $(e^{i\theta}, p, e^{i\theta}q)$ identifies    $J(\sigma\times \tau)$ with $J(\sigma)\times \tau$.  As an application of this, we have that $$\partial_J(\tau_1\times U_2)\sim U_1\times U_2 -U_1'\times U_2$$  and thus $U_1\times U_2$ is homologous to $U_1'\times U_2$.  Similarly, $U_1'\times U_2$ is homologous to $U_1'\times U_2'$. The desired conclusion now follows from pulling back $Ch(V_1\oplus V_2)$ to $M$ via the diagonal.     
\end{proof}  

\begin{theorem}
$Ch(V)$ agrees with the usual definition of the total Chern class. 
\end{theorem}
\begin{proof}
By the splitting principle \cite{BottTu}, it suffices to treat the case of a line bundle.  
\end{proof}
 
\subsection{Steenrod Operations and Equivariant Cohomology}
Let $V\rightarrow M$ be a real vector bundle. $V$ has a natural fiberwise $\mathbb{Z}_2$-action given by $$(m,v)\mapsto (m,-v)$$ 
As in the complex case, the Thom class $U$ is given by $i:M\rightarrow V$ where $i$ is the inclusion of $M$ as the zero section.  We view $U$ as a $\mathbb{Z}_2$-equivariant cohomology class in $\HHep^*(V; \mathbb{Z}_2)$.  In direct analogy with our construction of Chern classes, one may define the total Stiefel-Whitney class:
\begin{definition}
Let $SW(V)=i^*(U)\in H^*_+(M; \mathbb{Z}_2)$ be the total Stiefel-Whitney class.  
\end{definition}
The verification of the product formula is identical to the case of Chern classes as discussed above.  \\\\
We may also use equivariant cohomology to define Steenrod operations.  Consider the action of $\mathbb{Z}_2$ on  $M\times M$ that sends $(x,y)$ to $(y,x)$.  Let $$\Delta: M \rightarrow M\times M$$ be the inclusion of $M$ as the diagonal.  Given a cohomology class $\tau: Q\rightarrow M$, note that $\tau \times \tau $ defines an equivariant class in $H^*_+(M\times M; \mathbb{Z}_2)$.  
\begin{definition}

The total Steenrod square class is defined as $$Sq(\tau)=\Delta^*(\tau \times \tau)\in H^*_+(M; \mathbb{Z}_2)$$  

\end{definition}

The top degree part of $Sq(\tau)$ is just the cup product $\tau \cup \tau$.  The verification of the fundamental properties of the Steenrod operations, including the Cartan formula, is nearly identical to the case of Chern classes.  In addition, a neighborhood of the diagonal in $M\times M$ may be identified with the tangent bundle of $M$ in a manner that respects the $\mathbb{Z}_2$-action.  It follows that the total Stiefel-Whitney class of the tangent bundle coincides with the Steenrod square of the diagonal.

\section{Relation To Other Theories}
In this section we construct natural maps between $H_*$ and various forms of (co)homology found in the literature.  The fact that the natural maps induce isomorphisms follow from the standard Mayer-Vietoris arguments and thus we mostly restrict our attention to describing the relevant maps.  \\\\
The standard construction of singular homology $H^{sing}_*(M)$ is based on mappings $$\sigma: \Delta^n\rightarrow M$$ where $\Delta^n$ is an $n$-simplex.  If $M$ is smooth, one may restrict to smooth mappings (\cite{Bredon}).  By definition, each smooth $\sigma$ is a well defined chain in $M$.  Since the definition of the boundary maps are compatible, we have a well defined map $$H^{sing}_*(M)\rightarrow H_*(M)$$ 
There is a variant of singular theory that is based in mappings of cubes (see \cite{Serre}) $$\sigma: [0,1]^n\rightarrow M$$   To obtain the correct chain complex one must quotient out degenerate cubes.  By definition, these are mappings $\sigma$ that are independent of one of the coordinates.  Note that such cubes are also degenerate from our point of view.  Indeed, lets say that $$\sigma :[0,1]^n=[0,1]^{n-1}\times [0,1]\rightarrow M$$ is independent of the last coordinate.  It follows that $\sigma$ has small image.  In addition, $$\partial \sigma: \partial([0,1]^{n-1})\times [0,1]\sqcup [0,1]^{n-1} \times \partial [0,1]\rightarrow M$$ is a union of a chain with small image and a trivial chain.  If $H^{cube}_*(M)$ denotes cubical homology, we have defined a natural map $$H^{cube}_*(M)\rightarrow H_*(M)$$ 
Let $\Omega^*(M)$ be the chain complex of smooth differential forms on $M$.  Integration defines a map $$\Omega^*_{deRham}(M)\rightarrow C_*(M;\Ar)^*$$  To see that this map is well defined, we must check that a form $\omega$ integrates to zero on trivial chains and chains with small image.  Indeed, if $$\sigma:P\rightarrow M$$ has small image, $\sigma^*(\omega)$ vanishes since the rank of $\sigma$ is less than the dimension of $P$ at every point in the domain.  If $\sigma$ is trivial, let  $\phi:P\rightarrow P$ be an orientation reversing diffeomorphism that commutes with $\sigma$. We have $$\sigma^*(\omega)=\phi^*\circ \sigma^*(\omega)$$.  On the other hand, since $\phi$ reverses orientation, $$  \int_P \phi^*\circ \sigma^*(\omega)=-\int_P \sigma^*(\omega)$$  Therefore, $\int_P \sigma^*(\omega) =0$.  \\\\
Let us now discuss the equivariant situation.  Assume that $S^1$ acts smoothly on $M$.  Let $\xi$ denote the vector field induced by the infinitesimal action.  The Cartan construction gives us a deRham version of equivariant cohomology as follows.  Let $\Omega^*_{S^1}(M)$ be the complex of $S^1$-invariant forms.  Form a new complex $$(\Omega^*(M)\otimes \Ar[v],d_{S^1})$$ where $v$ has degree 2 and the differential acts as $$d_{S^1}(\omega)=d\omega+i_\xi \omega v$$  The resulting cohomology groups $H^{S^1, deRham}_*(M)$ are the Cartan model for the equivariant cohomology of $M$ \cite{Jones2}.   We define a chain map $$\Omega^*_{deRham}(M)\otimes \Ar[v]\rightarrow C^-_*(M;\Ar)^*$$ as follows.  Given a chain $\sum_{i=0}^N \sigma_i u^{-i} \in C_k^-(M)$ with $\sigma_i$ of dimension $k-2i$ and a cochain $$\sum_{i=0}^N \omega_i v^i\in \Omega^*_{deRham}(M)\otimes \Ar[v] $$ we have the evaluation map 
$$\sum_{i=0}^N \int_{P_i}\sigma_i^*(\omega_i)\in \Ar$$  To check that this defines a chain map, we use Stoke's theorem and the fact that $\int_{S^1\times P}\sigma^*(\omega)=\int_P\sigma^*(i_\xi \omega)$ as $\omega$ is $S^1$-invariant.

\section{Universal Coefficients}
The proof of the universal coefficient theorem for homology has a subtlety in our case.  As defined, it is not obvious that $C_*(M)$ is free.  However, we have:
\begin{lemma}
For any $M$, $C_*(M)$ is torsion-free.
\end{lemma}
\begin{proof}
Suppose there is $\sigma \in C_*(M)$ with $k\sigma \sim 0$ for some $k\in \mathbb{Z}^+$.  As in proof of lemma $\ref{lem32}$, we decompose $\sigma$ as $$\sigma=\sigma_1\sqcup \sigma_2 \dots $$ where each $\sigma_i$ consists of the union of isomorphic connected components (up to orientation) of $\sigma$ and no $\sigma_i$, $\sigma_j$ share isomorphic components when $i\neq j$. Since $k\sigma \sim 0$, we have that either $k\sigma_i$ is trivial or has small image for each $i$.  This implies that each $\sigma_i$ is either trivial or has small image.   We can ignore all the trivial components and assume that each $\sigma_i$ has small image and thus  $\sigma$ has small image.  By hypothesis, $k\partial \sigma \sim 0$.  By repeating the argument, it follows that $\partial \sigma$ is a union of a trivial chain and a chain with small image.  This implies that $\sigma \sim 0$ as it is degenerate. \end{proof}
By standard homological arguments (see \cite{Weibel}), $C_*(M)$ is thus a flat $\mathbb{Z}$-module.  Suppose $A$ is a $\mathbb{Z}$-module.  Let $H_*(M;A)$ denote the homology of $C_*(M)\otimes_{\mathbb{Z}} A$.  We may now use the standard homological arguments (see \cite{Weibel}) to prove the universal coefficients theorem:

\begin{theorem}
For any $\mathbb{Z}$ module $A$, we have the exact sequence $$0\rightarrow H_*(M; \mathbb{Z})\otimes A \rightarrow H_*(M;A)\rightarrow Tor^{\mathbb{Z}}(H_{*-1}(M; \mathbb{Z}),A)\rightarrow 0$$
\end{theorem}

\section{Duality}
Our definition of cohomology groups is geometric and does not follow the standard route of dualizing the  homology chain complex.  Recall that we have defined a cap product $$H_a(M;  \mathbb{Z})\otimes H^{b}(M; \mathbb{Z})\rightarrow H_{a-b}(M;  \mathbb{Z})$$ by taking intersections of generic transverse representatives for our cycles.  Letting $b=a$ and taking coefficients in a field $\mathbb{K}$, we get a map:
$$D:H^a(M;\mathbb{K})\rightarrow Hom(H_a(M),\mathbb{K})$$
\begin{theorem}
The map $D$ is an isomorphism.   
\end{theorem}

\begin{proof}
This is a variant of the standard Mayer-Vietoris argument (see \cite{BottTu}).  One covers $M$ by convex balls and inducts on the number of balls in a cover.  The inductive step consists of recognizing that for a open ball in $\Ar^n$, we have a perfect paring between the cycle represented by a point and the cocyle represented by the entire open ball. 
\end{proof}
\noindent \textbf{Remark.} Our version of geometric duality does not assume $M$ is compact.  Now, assume that $M$ is a compact oriented manifold of dimension $n$.  We have an identification $$H_*(M; \mathbb{Z})\cong H^{n-*}(M; \mathbb{Z})$$ sending a cycle $\sigma$ of dimension $k$ to the cocyle represented by $\sigma$ of codimension $n-k$. Note that an oriented chain may be viewed as a cooriented cochain using the orientation of $M$.  From our point of view, this last isomorphism is basically a tautology since we do not define cohomology groups by dualizing homology.  Therefore, the real content of duality lies in the nondegeneracy of the intersection pairing.


\section{Spaces with a Functional}
\subsection{Basic Definitions}
So far, we have worked exclusively with smooth manifolds and smooth maps.  In a sense, this is not a great loss of generality since any finite CW complex has the homotopy type of a smooth manifold (possibly noncompact).  However, there are certain natural geometric constructions that are not possible if one restricts to smooth spaces.  For example, given a subset $S\subset M$ one may form the quotient $M/S$.  However, even if $S$ is a smooth submanifold, $M/S$ does not in general inherit a smooth structure.  For instance, the construction of cones and suspensions that is central to homotopy theory  leads to singular spaces.  One of the goals of this section is to provide a language to model these nonsmooth phenomena in the smooth category.  This leads us to the notion of a space with functional and morphisms between such spaces.  This formalism plays a crucial role in the construction of Floer theory in \cite{Lipyan} and we hope it is also illuminating in finite dimensions.   
\begin{definition}
A \textbf{space with a functional} $(M,f)$, is a pair where $M$ is a smooth finite dimensional manifold without boundary and $$f:M\rightarrow \Ar$$ is a continuous function.  
\end{definition}
The purpose of introducing $f$ is that it provides a way of characterizing the points in $M$ that approach $\pm\infty$.  We come to the definition of homology groups of $(M,f)$:
\begin{definition}
Let $P$ be a smooth oriented manifold with corners.  A smooth map $$\sigma:P\rightarrow M$$ is a chain if given $p_i\in P$, either $p_i$ has a converging subsequence or $$\lim_{i\rightarrow \infty} f(\sigma(p_i))=-\infty$$   
\end{definition}  
Intuitively, the only source of noncompactness of $\sigma$ comes from points going to $-\infty$.  We may now define trivial and degenerate chains as in our definition of geometric homology in section $\ref{homdef}$.  We let $H_*(M,f)$ denote the resulting groups.   They are graded by the dimension of the chains. \\\\
\textbf{Example 1.} Consider $(\Ar, f)$ for several different choices of $f$.  If $f=0$, we recover the usual homology of $\Ar$.  If $f(t)=t$ then in fact $H_*(\Ar, t)=0$.  Indeed, given a point $m\in \Ar$, we may form the homotopy $$(-\infty,0]\rightarrow \Ar$$
 mapping $t\in (-\infty, 0]$ to $m+t$.  The same homotopy shows that all 1-cycles are trivial as well. If we take $f(t)=-t^2$, we have $H_0(\Ar, -t^2)=0$ by a similar argument.  On the other hand, this time $H_1(\Ar, -t^2)=\mathbb{Z}$.  It is generated by the identity map $\Ar \rightarrow \Ar$.  This cycle is nontrivial since it pairs with a point to give 1.\\\\
\textbf{Example 2.} Given a vector bundle $V\rightarrow M$, fix a metric on $V$.  Let $f(m,v)=-|v|^2$. One may view $(V, f)$ as the smooth analogue of the Thom space construction.  Note that any fibre defines a homology class in $H_*(V, f)$, provided $V$ is oriented.  \\\\
\textbf{Example 3.}  Let $w$ be a closed $1$-form on $M$ and let $\widetilde{M}$ be a cover.  Suppose that the pullback of $w$ is an exact 1-form $df$.  Consider $(\widetilde{M}, f)$.  $H_*(\widetilde{M}, f)$ is the topological analogue of Novikov homology (see \cite{Farber}).  Indeed, in addition to compact cycles, we allow for countable sums of compact cycles as long as the images tend to $-\infty$.  \\\\
\textbf{Example 4.}  We say that $(M,f)$ is \textbf{pointed} if each component of $M$ contains a sequence that tends to $-\infty$.  For such spaces, $H_0(M,f)$ vanishes since a point in $M$ may be moved to $-\infty$ by a path $\gamma: (-\infty, 0] \rightarrow M$.  Given any manifold $M$, we may form a pointed space by removing a point $p_i$ from each component $M_i$ and taking $$f_i:M_i-p_i\rightarrow \Ar$$ to be a function that approaches $-\infty$ as points approach $p_i$.  For instance, starting from $S^1$ we  obtain $(\Ar, -t^2)$ by this procedure.  Given two pointed spaces $(M,f)$ and $(N,g)$ we have the product $(M\times N,f+g)$.  This is actually the smash product for pointed spaces. For example,  $(\Ar^n,-|x|^2)$  is the $n$th iterated product of $(\Ar, -t^2)$.  It is the pointed version of $S^n$.  \\\\
\textbf{Example 5.} Our formalism makes its possible to construct cones in the smooth category.  Let $S\subset M$ be a closed submanifold.  Take a function $$f:M\rightarrow [0,1]$$ that equals 1 in a tubular neighborhood $V$ of $S$ and vanishes outside a larger tubular neighborhood $U$. Let $I=(-\infty, -1)$.   Consider $(M\times I, f\cdot t)$ where $t$ is the coordinate on $\Ar$.  We decompose our space $M\times I$ as $(M-U)\times I$ and $U\times I$.  On $(M-U)\times I$,  $f\cdot t=0$ and thus, up to homotopy, we get $M-U$.  On $U\times I$, on the other hand, we claim that every cycle is nulhomotopic.  Indeed, given a cycle $P$, first deform it to lie in $V\times I$,  Now, observe that by  moving the $t$ coordinate to $-\infty$ we provide a homotopy from our cycle to the empty set.  \\\\
We also have the notion of a cochain:
\begin{definition}
Let $Q$ be a smooth manifold with corners.  A \textbf{cochain} $\tau:Q\rightarrow M$ is a smooth cooriented map such that given $q_i\in Q$, either $q_i$ has a converging subsequence or $\lim f(\tau(q_i))=\infty$.  
\end{definition}
We let $H^*(M,f)$ denote the resulting cohomology groups which are graded by codimension.  Note that we have pairing $$H^*(M, f)\otimes H_*(M,f)\rightarrow H_*(M)$$ since the intersections of a cycle and a cocycle must be compact. \\\\
The notion of cohomology of a  space with functional is a generalization of usual notion of cohomology.  Namely, given any smooth manifold $M$, let $f$ be any exhausting function on $M$ that is bounded below.  Thus, $m_i\in M$ either has a convergent subsequence or $\lim_{i \rightarrow \infty} f(m_i)=\infty$.  We associate to  $M$ the space $(M,f)$.  In this case, we have:
\begin{lemma}
We have $H_*(M,f)\cong H_*(M)$ and  $H^*(M,f)\cong H^*(M)$
\end{lemma} 
\textbf{Remark.} Note that if we take $f=0$ in the previous construction, we get cohomology with compact supports.

\subsection{Morphisms}
\begin{definition}
Let $U$ be open in $M$.  A morphism $\phi:(M,f)\rightarrow (N,g)$ is a smooth map $$\phi:U \rightarrow N$$ such that if $m_i\in U$ does not have a convergent subsequence in $U$, then $g(\phi(u_i))\rightarrow -\infty$.  
\end{definition}

\begin{lemma}
Let $\phi$ be a morphism.  Given a chain $\sigma:P\rightarrow M$, $$\phi\circ \sigma_{|\sigma^{-1}(U)}:\sigma^{-1}(U)\rightarrow N$$ is a chain.
\end{lemma}
\begin{proof}
Suppose $p_i\in \sigma^{-1}(U)$ is a sequence without a convergent subsequence.  We claim that $\sigma(p_i)$ does not have a convergent subsequence in $U$.  Indeed, if some subsequence $\sigma(p_j)$ had a limit $y\in U$, then $p_j$ must have a convergent subsequence $p_k \in P$.  But this implies that  $\lim p_k \in  \sigma^{-1}(U)$, contradicting the assumption on $p_i$.  Therefore, since $\sigma(p_i)$ does not have a convergent subsequence in $U$, $g(\phi(\sigma(p_i))) \rightarrow -\infty$  as desired.    
\end{proof}
\noindent \textbf{Example 1.}  Let $S\subset M$ be a compact submanifold.  Let $$f:M-S\rightarrow \Ar$$ be a function that is bounded outside a neighborhood of $S$ and approaches $-\infty$ as a point approaches $S$.  We define the collapsed space to be $(M-S, f)$.   The morphism $$\phi: M\rightarrow M-S$$   defined by the identity map on $M-S$ is our smooth analogue of the collapsing map.  \\\\
An important special case that is central to homotopy theory is the collapsing map $$S^n\rightarrow S^n \vee S^n$$ that is obtained by collapsing an equatorial $S^{n-1}$.  Here we view $S^n \vee S^n$ as a disjoint union of two open hemispheres of the original $S^n$.


\end{document}